\documentclass{article}
\usepackage[margin=1in]{geometry}
\usepackage{verbatim}
\usepackage{authblk}
\usepackage{todonotes}

\usepackage{amsmath}
\usepackage{amssymb}
\usepackage{amsthm}
\usepackage{amscd}
\usepackage{enumerate}
\usepackage[colorlinks=true]{hyperref}
\usepackage{bbm}
\usepackage{etoolbox}

\usepackage[utf8]{inputenc}

\newcommand{\tomemail}{\href{mailto:tom.bachmann@zoho.com}{tom.bachmann@zoho.com}}

\ifdefvoid{\longdocument}{
\newtheorem*{defn}{Definition}
\newtheorem{prop}{Proposition}
}{
\newtheorem{prop}{Proposition}[chapter]

}
\newtheorem{corr}[prop]{Corollary}
\newtheorem{lemm}[prop]{Lemma}
\newtheorem{thm}[prop]{Theorem}

\newtheorem*{thm*}{Theorem}
\newtheorem*{corr*}{Corollary}
\newtheorem*{prop*}{Proposition}
\newtheorem*{lemm*}{Lemma}
\theoremstyle{definition}
\newtheorem{defnn}[prop]{Definition}

\theoremstyle{remark}
\newtheorem{rmk}[prop]{Remark}

\newcommand{\Aone}{{\mathbb{A}^1}}
\newcommand{\Gm}{{\mathbb{G}_m}}

\newcommand{\id}{\operatorname{id}}
\newcommand{\ZZ}{\mathbb{Z}}

\newcommand{\NN}{\mathbb{N}}

\newcommand{\FF}{\mathbb{F}}

\newcommand{\SH}{\mathbf{SH}}

\newcommand{\iHom}{\ul{\operatorname{Hom}}}
\newcommand{\Hom}{\operatorname{Hom}}

\newcommand{\colim}{\operatorname{colim}}

\newcommand{\Mod}{\textbf{-Mod}}
\newcommand{\Alg}{\textbf{-Alg}}

\DeclareRobustCommand{\ul}{\underline}

\newcommand{\tunit}{\mathbbm{1}}

\newcommand{\iso}{\cong}
\newcommand{\wequi}{\simeq}

\hypersetup { colorlinks = false }

\title{Some Remarks on Units in Grothendieck--Witt Rings}
\author{Tom Bachmann\footnote{present address: Fakultät Mathematik,
               Universität Duisburg-Essen, Thea-Leymann-Straße 9,
	       45127 Essen, Germany} \\
	LMU Munich, Munich, Germany \\ \tomemail}
\newcommand{\Fet}{{F\acute{e}t}}
\newcommand{\lra}[1]{\langle #1 \rangle}
\newcommand{\Sym}{\operatorname{Sym}}

\begin{document}

\maketitle

\begin{abstract}
We establish new structures on Grothendieck--Witt rings, including
a $GW(k)$-module structure on the unit group $GW(k)^\times$ and a
presentation of $\ul{GW}^\times$ as an infinite $\Gm$-loop sheaf.
Even though our constructions are motivated by speculations in stable
$\Aone$-homotopy theory, our arguments are purely algebraic.
\end{abstract}

\tableofcontents

\section{Introduction}

The main objects of investigation of this article are the ring-valued
functors $X \mapsto GW(X)$ and $X \mapsto \ul{GW}(X)$ and their subfunctors of
units $GW^\times(X)$ and $\ul{GW}^\times(X)$.
Recall that for a scheme $X$, $GW(X)$ is the \emph{Grothendieck--Witt}
ring of
$X$ \cite{knebusch-bilinear}, and for $X$ smooth over a perfect field,
$\ul{GW}(X)$ is the \emph{unramified Grothendieck--Witt ring} of $X$
\cite[Chapter 3]{A1-alg-top}. The connection is that for $X$ (essentially)
smooth local, we have $GW(X) = \ul{GW}(X)$, cf. \cite[Theorem A]{panin-witt-purity}.

Our principal contribution is the following. We show that if $k$ is a field of
characteristic not $2$, then the group of units $GW^\times(k)$ has a
canonical structure of a module over $GW(k)$, related to Rost's multiplicative
transfer on $GW(k)$. We use this to give a novel presentation of $GW^\times(k)$,
see Proposition \ref{prop:GW-times-presentation},
and to construct a \emph{homotopy module} $T_*$ such that $T_0 \iso
\ul{GW}^\times$. See Appendix \ref{sec:app-htpy-mod} for some recollections regarding
homotopy modules.

\paragraph{Organisation.}

We now provide an overview of the article. The remaining subsections of the introduction
provide a more leisurely account of some of the key ideas mentioned here.

In Section \ref{sec:mult-transfer}, we recall the results of Rost and his
students on multiplicative transfers for the Grothendieck--Witt ring $GW(X)$
\cite{rost2003multiplicative,krsnik2006multiplikative}. Specifically, the
multiplicative transfer of Rost is defined using a certain norm functor for modules,
also defined by Rost. We show that Rost's norm construction coincides with
a more general construction of Ferrand \cite{ferrand1998foncteur}, in the
situation where both apply.

In Section \ref{sec:tambara}, using this comparison of norm constructions,
we show that the assignment $\Fet/S \ni X \mapsto GW(X)$
defines a Tambara functor. Here $\Fet/S$ denotes the category of finite étale
schemes over $S$, and by a Tambara functor on this category we mean the evident
extension of the notion from \cite{tambara1993multiplicative}; see Definition
\ref{def:tambara} for details. Using a result of Tambara \cite[Theorem 6.1]{tambara1993multiplicative},
this also yields an alternative proof that the
norm maps extend from $Iso(Bil(\bullet))$ to $GW(\bullet)$.

Section \ref{sec:GW-module-structure} contains our main observation. We show
that if $k$ is a field of characteristic not 2, then the group of units
$GW^\times(k) \subset GW(k)$ is a module over $GW(k)$, in a unique way that is
compatible with the projection formula. By this we mean that if $A/k$ is finite
étale, then for $x \in GW(A)$ and $y \in GW^\times(k)$ the following formula
holds:
\[ y^{tr_{A/k}(x)} = N_{A/k}((y|_A)^x). \]
Note that we write the module structure as ``exponentiation''.
This result is Proposition \ref{prop:GW-module-existence-basics}.
Uniqueness of the $GW(k)$-module structure follows from the fact that as an
abelian group,
$GW(k)$ is generated by the traces of finite étale algebras, in fact traces of
degree at most 2 extensions suffice. This is explained before Proposition
\ref{prop:GW-module-existence-basics}. Existence/well-definedness is a consequence of
Serre's splitting principle; see Lemma \ref{lemm:main-observation}. In the
remainder of that section we establish many simple but useful properties of this
$GW(k)$-module structure.

In Section \ref{sec:sheaf-GW-times} we pass to associated sheaves. Thus we
study the unramified sheaf of units $\ul{GW}^\times \subset \ul{GW}$. We define
a filtration $F_\bullet \ul{GW}^\times$, where $x \in F_n$ if and only if $x
\equiv 1 \pmod{\ul{I}^n}$. Here $\ul{I}$ denotes the unramified sheaf of
fundamental ideals. We can determine the subquotients $F_n
\ul{GW}^\times/F_{n+1} \ul{GW}^\times$ and this allows us to prove in Theorem
\ref{thm:GW-times-strictly-htpy-inv} that
$\ul{GW}^\times$ is strictly homotopy invariant. This fixes a gap in a result
of Wendt \cite{wendt-units}. With this preliminary out of the way, we can extend
the module structure from the previous section to obtain a $\ul{GW}$-module
structure on $\ul{GW}^\times$. Using it we define morphisms $\beta_n^\dagger: F_n
\ul{GW}^\times \to (F_{n+1} \ul{GW}^\times)_{-1}$. In Proposition
\ref{prop:betan-iso} we prove that $\beta_n^\dagger$ is an isomorphism
for $n \ge 2$. In other words, we have
constructed a homotopy module $F_*$ with $F_n = F_n \ul{GW}^\times$ for $n \ge
2$.

In Section \ref{sec:logarithm} we study the homotopy module $F_*$. We show that
for $* \ge 2$
its canonical $\ul{GW}$-module structure and cohomological transfers coincide
respectively
with the module structure on $F_n \ul{GW}^\times$ constructed in
the previous section and Rost's multiplicative transfers. In doing so we define
an isomorphism of homotopy modules $\log: F_* \to \ul{I}^*_{tor}$. In
particular this map turns Rost's multiplicative transfers into the usual
additive ones.

In the final Section \ref{sec:delooping} we put everything together in Theorem
\ref{thm:final}. There we construct
a homotopy module $T_*$ with $T_0 \iso \ul{GW}^\times$ such that the
$\ul{GW}$-module structure and cohomological transfers on $T_0$ correspond
respectively to the $\ul{GW}$-module structure and multiplicative transfers on
$\ul{GW}^\times$ we have constructed before. Along the way we establish the
following exact sequence:
\begin{equation*} \label{eq:GW-times-presn}
  0 \to \ul{I}^2/2\ul{I} \xrightarrow{i \oplus \lra{2} - 1} \ul{GW}/2 \oplus
        \ul{I}^2_{tor} \to \ul{GW}^\times \to 1.
\end{equation*}
Here $i$ is induced from the canonical inclusion $\ul{I}^2 \to \ul{GW}$.
This appears to be a novel presentation of the group of units of $\ul{GW}$.

The paper concludes with two short appendices. In Appendix
\ref{sec:app-htpy-mod} we recall the basics about homotopy modules. In Appendix
\ref{sec:app-cont} we recall a well-known continuity result.

\paragraph{Notations and conventions.}

We make the blanket assumption throughout that $k$ is a field of characteristic
different from $2$. When dealing with homotopy modules, we also assume that the
base field is perfect, since this is when the theory is most well-behaved.

If $X$ is a scheme and $a \in \mathcal{O}^\times(X)$, we denote by $\lra{a} \in
GW(X)$ the class of the bilinear space with underlying vector bundle $\mathcal{O}_X$
and with form $(x, y) \mapsto axy$.

\paragraph{Acknowledgements.}

This work would not have been possible without Rost's initiation of the study of
multiplicative transfers on $GW$ \cite{rost2003multiplicative} and the detailed
computations by his student Wittkop in the case of a quadratic extension
\cite{wittkop2006multiplikative}. The unpublished preprint of Wendt
\cite{wendt-units} also was very influential to us, even though we do not end up
explicitly re-using any of the results from that paper.

This work has been carried out while the author was a post-doctoral student at LMU
Munich.

The author would like to thank an anonymous referee for an exceptionally
thorough review, as a result of which the presentation was improved considerably.

\paragraph{Motivation from $\Aone$-homotopy theory.}
If $E$ is an $E_\infty$-ring spectrum (in the classical sense), then
$\Omega^\infty E$ is an $E_\infty$-ring space, and hence the subspace of units
$(\Omega^\infty E)^\times$ is a grouplike $E_\infty$-space. Consequently, by
classical infinite loop space theory, there
is a (unique) connective spectrum $gl_1(E)$ such that $\Omega^\infty gl_1(E)
\wequi (\Omega^\infty E)^\times$ as $E_\infty$-spaces.

We are motivated by the question if a similar construction can be possible in
stable motivic homotopy theory \cite[Section 5]{morel2004motivic-pi0}. We thus
fix a base field $k$, and we have the adjunction
\[ \Sigma^\infty_+: \mathbf{H}(k) \leftrightarrows \SH(k): \Omega^\infty, \]
where $\mathbf{H}(k)$ denotes the unstable, unpointed motivic homotopy category over
$k$, and $\SH(k)$ denotes the $\mathbb{P}^1$-stable motivic homotopy category
over $k$. One may show that if $E \in \SH(k)$ is an $E_\infty$-ring
spectrum in the ordinary sense, then it need not be the case that there exists
$gl_1(E) \in \SH(k)$ such that $\Omega^\infty(E)^\times \wequi \Omega^\infty
gl_1(E)$, even just as motivic spaces. In other words, the space of units in $E$
need not be an infinite $\mathbb{P}^1$-loop space. This does not rule out, however, that
there might be a stronger notion of a highly structured ring spectrum in
$\SH(k)$ which has this property.\footnote{In fact, in \cite{bachmann-norms} the
author and M. Hoyois introduce the notion of \emph{normed spectra}, which are
enhancements of naive $E_\infty$-algebras in motivic spectra. If $E$ is a normed
spectrum then $\ul{\pi}_0(E)$ acquires multiplicative transfers along finite
étale morphisms. The sphere spectrum $S^0$ has a unique normed structure, and
the multiplicative transfers induced on $\ul{\pi}_0(S^0) = \ul{GW}$ are Rost's
norms \cite[Theorem 10.13]{bachmann-norms}. Unfortunately we do not know if $E$
being a normed spectrum implies that $(\Omega^\infty E)^\times$ is an infinite
$\mathbb{P}^1$-loop space.}

Being an infinite $\mathbb{P}^1$-loop space is a very strong requirement. We are
only studying one obstruction in this article: if $X$ is an infinite
$\mathbb{P}^1$-loop space, then all the homotopy sheaves $\ul{\pi}_i(E)$ extend
to \emph{homotopy modules}. Among many other things, this means that they must
be modules over $\ul{\pi}_0(\Omega^\infty S^0) = \ul{GW}$, cf. \cite{morel2004motivic-pi0}.

Hence the starting point of our investigation: if there is any notion of a highly
structured motivic commutative ring spectrum, surely the sphere spectrum must be
an example of such an object. As cited above, $\ul{\pi}_0(S) = \ul{GW}$ is the
sheaf of unramified Grothendieck--Witt groups. Hence if there is to be any hope for
a motivic multiplicative infinite loop space theory of this form, the sheaf of units
$\ul{GW}^\times$ must extend to a homotopy module. In particular, for every
field $k$, the group of units in the Grothendieck--Witt ring of $k$ must be a module
over the Grothendieck--Witt ring itself!

The aim of this article is to show that this is indeed the case, and that in fact
$\ul{GW}^\times$ does extend to a homotopy module. Hence, at least
from this perspective, the existence of a motivic multiplicative infinite loop
space theory is not ruled out.

\paragraph{The $GW$-module structure.}

At the first sight, the claim that $GW^\times(k)$ should be a module over $GW(k)$
may seem preposterous; at least it did so to the author. Here we try to de-mystify
this structure somewhat. First a philosophical remark: one should think of the
$GW$-module structure as arising in essentially the same fashion as the
$\ZZ$-module structure on $\ZZ^\times = \{\pm 1\}$. For this reason, we write
the action of $x \in GW(k)$ on $y \in GW^\times(k)$ as the ``exponentiation''
$y^x$.

Secondly, here are some
formulas. The defining property of this $GW$-module
structure is that for $x \in GW^\times(k)$ and $A/k$ finite étale, we have
$x^{tr(A)} = N_{A/k}(x|_A)$. Here $tr_{A/k}: GW(A) \to GW(k)$ is the \emph{Scharlau
transfer}, $N_{A/k}: GW(A) \to GW(k)$ is the \emph{Rost Norm}
\cite{rost2003multiplicative}, and $tr(A) := tr_{A/k}(1)$.
Suppose that $A = k(\sqrt{a})$. Then one may
check that $tr(A) = \lra{2}(1 + \lra{a})$. From this it follows easily that
elements of the form $tr(A)$ for $[A:k] \le 2$ generate $GW(k)$ as an
abelian group, so we only have to understand these exponents. Fortunately this
situation has been studied thoroughly by Wittkop
\cite{wittkop2006multiplikative}, and the following formula is an immediate
corollary of his work (see Lemma \ref{lemm:wittkop-formula}):
\[ (x + y)^{tr(A)} = x^{tr(A)} + y^{tr(A)} + tr(A)xy. \]
Here $A/k$ is an extension of degree 2. The trivial extension $A = k \times k$
is allowed, in which case $tr(A) = 2 \cdot \lra{2} = 2$ (see e.g. Lemma
\ref{lemm:2-vanishing}) and the formula is familiar.
From this (and $N_{A/k}(0) = 0$) one also obtains
\[ (-1)^{tr(A)} = tr(A) - 1; \]
see Proposition \ref{prop:norm-not-so-basics} part (ii).
Finally one may check the following formula for all $x \in GW(k), a \in k^\times$:
\[ \lra{a}^x = \lra{a}^{dim(x)}. \]

Together these three formulas in principle allow the computation of $x^y$ for
any $x \in GW^\times(k)$ and $y \in GW(k)$. This is illustrated for example in
the proof of Proposition \ref{prop:norm-not-so-basics}.

\paragraph{The logarithm isomorphism.}

A further surprising property is that at least on some part of $GW^\times(k)$,
the multiplicative structures can be made equivalent to the additive ones. Let
us also try to shed some light on that.

The logarithm map furnishes an isomorphism of abelian groups
\[ \log: F_2 GW^\times(k) = (1 + I^2_{tor}(k), \times) \to (I^2_{tor}(k), +). \]
This map satisfies
\[ \log(xy) = \log(x) + \log(y), \text{ for } x, y \in F_2 GW^\times(k), \]
\[ \log(x^z) = z\log(x), \text{ for } z \in GW(k) \]
and
\[ \log(N_{A/k}(w)) = tr_{A/k}(\log(w)), \]
for $A/k$ finite étale and $w \in F_2 GW^\times(A)$. These three properties are
what we mean by turning multiplicative structures into additive ones.

The logarithm map is constructed as follows. Given $x \in F_2 GW^\times(k)$, let $t_1, 
\dots, t_m$ be independent variables. Then consider the element
\[ y := x^{(\lra{t_1} - 1) \dots (\lra{t_m} - 1)} \in GW^\times(k(t_1, \dots, t_m)). \]
It follows from the theory of homotopy modules that $y$ may be written as $y = 1
+ \log_{(m)}(x) (\lra{t_1} - 1) \dots (\lra{t_m} - 1)$ for a unique element
$\log_{(m)}(x) \in I^2_{tor}(k)$; see Lemma \ref{lemm:defn-log-approx}. Then
$\log(x) := \lim_{m \to \infty} \log_{(m)}(x)$. This limit makes sense because
the sequence is eventually constant; see Theorem
\ref{thm:log-module-transfers}. See also Remark \ref{rmk:classical-analysis-log}
for a comparison to the logarithm function in real analysis.

\paragraph{Remark on characteristic 2.}
The theory of Grothendieck--Witt rings in characteristic 2 can be rather
different from the other characteristics. However, when working with
\emph{bilinear} forms, many of these differences disappear. A natural question
is then if our results can be extended to characteristic 2. The main obstruction
to this is the following:
\begin{lemm}
Let $char(k) = 2$ and $l/k$ be a finite separable extension. Then the bilinear
form $tr_{l/k}(1) \in GW(k)$ is isomorphic to the trivial form $[l:k]$.
\end{lemm}
Indeed, recall that we try to define a $GW(k)$-module structure on
$GW(k)^\times$ by requiring that $u^{tr_{l/k}(1)} = N_{l/k}(u|_l)$; the above
Lemma shows that in characteristic 2 this formula does not put any constraints
on a hypothetical module structure at all. Thus our method cannot work in
characteristic 2.
\begin{proof}
If $[l:k]=2$ this is checked by direct computation. In general, after an odd
degree base change, which induces an injection on $GW$,
we may assume that $l/k$ is obtained as a sequence of quadratic extensions; the
result follows. See Proposition \ref{prop:norm-basics} for a more detailed
proof using a similar argument.
\end{proof}

\section{Multiplicative transfers on $GW$}
\label{sec:mult-transfer}

\subsection{Split $K$-groups}
Given a scheme $X$, we have the categories $Vect(X)$ and $Bil(X)$ of
vector bundles on $X$ and vector bundles on $X$ provided with a
bilinear form, respectively. By ``bilinear form'' we shall always mean a symmetric,
non-degenerate bilinear form. In other words an object of $Bil(X)$ is a vector
bundle $E$ together with a homomorphism $b_E: E \otimes E \to \mathcal{O}_X$, such
that (1) $b_E$ is symmetric, i.e. $b_E \circ \tau_E = b_E$, where
$\tau_E: E \otimes E \to E \otimes E$ is the twist isomorphism, and (2) the
homomorphism
$b^\vee_E: E \to E^\vee$ induced by adjunction is an isomorphism. Here $E^\vee$
denotes the dual bundle.

Write $Iso(Vect(X))$ for the set of isomorphism classes of vector
bundles; this is an abelian semi-group. Let $K(Vect(X))^\oplus$ be its
associated Grothendieck group. It is also known as the direct-sum $K$-theory
$K_0^\oplus(X)$ of $X$. In other words $K(Vect(X))^\oplus$ is obtained as the
Grothendieck group of the exact category $Vect(X)$, but where only split exact
sequences are allowed in the exact structure. If $X$ is affine, this coincides with the usual group
$K_0(X)$ (where all exact sequences are allowed in the exact structure),
but for general $X$ it does not. We can do the same with $Bil(X)$: we get the abelian semi-group
$Iso(Bil(X))$, and the associated Grothendieck group $K(Bil(X))^\oplus$ coincides with
the usual Grothendieck--Witt group $GW(X)$ for $X$ affine.

If $f: X \to Y$ is a morphism of schemes, there is the usual pushforward $f_*:
QCoh(X) \to QCoh(Y)$. If $f$ is finite locally free (see e.g. \cite[Tag
02KA]{stacks-project} for a definition) and $V \in QCoh(X)$ is a vector
bundle, then $f_*(V) \in QCoh(Y)$ is also a vector bundle. Thus there is an
induced map $f_*: Iso(Vect(X)) \to Iso(Vect(Y))$. Since $f_*(E \oplus F) \iso
f_*(E) \oplus f_*(F)$, this descends to the Grothendieck group to yield a
push-forward homomorphism $tr_f := f_*: K(Vect(X))^\oplus \to K(Vect(Y))^\oplus$.

If in addition $f$ is étale and $E \in Bil(X)$
then the trace map $f_* \mathcal{O}_X \to \mathcal{O}_Y$ can be used to turn
$f_*E \in Vect(Y)$ into a bilinear bundle. Indeed we let $b_{f_* E}$ be the
composite
\[ f_* E \otimes f_* E \to f_*(E \otimes E) \xrightarrow{f_* b} f_*
\mathcal{O}_X \to \mathcal{O}_Y. \]
Here we have used that $f_*$ is right adjoint to a symmetric monoidal functor,
so is lax symmetric monoidal. Then as before we obtain $tr_f:
K(Bil(X))^\oplus \to K(Bil(Y))^\oplus$.

The tensor product of bundles turns $Iso(Vect(X))$
and $Iso(Bil(X))$ into semi-rings, and $K(Vect(X))^\oplus$ and
$K(Bil(X))^\oplus$ into
rings. However, $tr_f$ is not a ring homomorphism: it does not respect
multiplication.
This can be remedied to some extent by considering a multiplicative version of
transfer. In the next two subsections, we explain two constructions of such
multiplicative transfers. In the last subsection, we compare the two.

\subsection{The Ferrand norm}

Given a finite locally free ring homomorphism $R \to S$, in
\cite{ferrand1998foncteur} D. Ferrand defined a norm functor $N_{S/R}: S\Mod \to
R\Mod$. It is lax symmetric monoidal \cite[(N6)]{ferrand1998foncteur} and so
preserves algebras, modules over algebras, etc.

We briefly review Ferrand's definition. If $M$ is an
$R$-module, we denote by $\ul{M} \in Fun(R\Alg, Sets)$ the functor
$R' \mapsto M \otimes_R R'$.

\begin{defnn}[\cite{ferrand1998foncteur}, 2.2.1]
A \emph{polynomial law} between
$R$-modules $M_1, M_2$ is a natural transformation $\underline{M}_1 \to
\underline{M}_2$.
\end{defnn}

Recall that if $M$ is a (finitely generated) locally free $R$-module and
$\alpha: M \to M$ is an endomorphism, one can define the \emph{determinant}
$det(\alpha) \in R$. Indeed
if $M$ has rank $n$ then the maximal exterior power
$\Lambda^n(\alpha): \Lambda^n M \to \Lambda^n M$ is an
endomorphism of the invertible rank 1 module $\Lambda^n M$, so must be given by
a multiplication by a (unique) element $det(\alpha) \in R$. If $M$ is not of
constant rank then $R = R_1 \times R_2 \times \dots \times R_k$ such that $M =
M_1 \times \dots \times M_k$ with each $M_i$ of constant rank, and we let
$det(\alpha) = (det(\alpha_1), \dots, det(\alpha_k))$. It is easy to see that
this is independent of the decomposition $R = \prod_i R_i$.

Since $S/R$ is finite locally free by assumption, we thus have the \emph{norm map}
$n_{S/R}: S \to R, x \mapsto det(\times
x: S \to S)$. Here $\times x: S \to S$ denotes the endomorphism $s \mapsto xs$.
This defines a polynomial law from $S$ to $R$, because the determinant
commutes with base change.

\begin{defnn}[\cite{ferrand1998foncteur} Definition 3.2.1] \label{defn:norm-law}
If $F \in S\Mod$ and $E \in R\Mod$, a \emph{norm law} from $F$ to $E$ is a
polynomial law $\phi$ from $F$ (viewed as an $R$-module) to $E$ such that for every
$R$-algebra $R'$ and $s \in R' \otimes_R S, x \in R' \otimes_R F$ we have
$\phi(sx) = n_{R' \otimes_R S/R'}(s)\phi(x)$.
\end{defnn}

Given $F \in S\Mod$, there is a universal norm law $n_F: F \to N_{S/R}(F)$,
in the following sense: if $E \in
R\Mod$ and $\phi$ is a norm law from $F$ to $E$, then there is a unique
$R$-homomorphism $u: N_{S/R}(F) \to E$ such that $\phi = u \circ n_F$
\cite[Theorem 3.2.3]{ferrand1998foncteur}. This defines the norm functor.

In geometric language, given a finite locally free morphism $f: X \to Y$ of
affine schemes, we have found a functor $N_f: QCoh(X) \to QCoh(Y)$. Although we
will not need this, let us remark that it is easy to see that the norm
construction can be extended to any finite locally free morphism of schemes, affine or
not.

\subsection{The Rost norm}
Ferrand's norm functor has nice technical properties (like being defined by a
universal property), but it can be difficult to get a
hold of computationally. There is an alternative construction due to Rost
\cite[3.2]{rost2003multiplicative}. Again we review the definition briefly.

If $M$ is any $R$-module, we define the $n$-th \emph{symmetric power} as
\[ \Sym^n_R(M) = (M^{\otimes n})^{\Sigma_n}, \]
where $\Sigma_n$ is the symmetric group on $n$ letters, acting by permuting
the factors of $M^{\otimes n}$, and the superscript means passing to invariants.
Note that if $S$ is an $R$-algebra, then so is $S^{\otimes n}$, and $\Sym^n S
\subset S^{\otimes n}$
is a subalgebra.

Now let $R \to S$ be locally free of rank $n$. Since $\Lambda^n_R S$ is a
quotient of $S^{\otimes n}$ and $\Sym^n_R S$ is a subalgebra of $S^{\otimes n}$,
$\Lambda^n_R S$ is a $\Sym^n_R S$-module. We thus get a canonical homomorphism
\[ \bar{n}_{S/R}: \Sym^n_R S \to End_R(\Lambda^n S) \iso R,\,\, x \mapsto (\times
x:\Lambda^n S \to \Lambda^n S).  \]
Here we have used that $End_R(\Lambda^n S) \iso R$ because $\Lambda^n S$ is an
invertible $R$-module.

\begin{defnn}[\cite{rost2003multiplicative}, 2.3]
Let $M$ be an $S$-module. We put
\[ \nu_{S/R}(M) = \Sym^n_R M \otimes_{\Sym^n_R S} R, \]
where the map $\Sym^n_R S \to R$ is $\bar{n}_{S/R}$ .
For more details, see \cite{krsnik2006multiplikative}.
\end{defnn}

Reverting to geometric language, for a finite étale morphism of constant rank
$f: X \to Y$ between affine schemes, we have found a functor $\nu_f: QCoh(X) \to
QCoh(Y)$. Again, even though we do not need the extra generality, it is easy to
see that this construction extends to any finite étale morphism, not necessarily
of constant rank, and not necessarily between affine schemes.

\subsection{Comparison of norms}

We now show that the constructions from the previous two subsections coincide in
sufficiently good cases.

\begin{prop}
Let $f: X \to Y$ be a finite étale morphism of schemes.
Then for $M \in Vect(X)$ there is a canonical isomorphism
$\nu_f(M) \iso N_f(M)$.
\end{prop}
\begin{proof}
Since the isomorphisms are canonical, i.e. compatible with base change in $Y$,
they can be glued in open covers. Thus
we may assume that $Y = Spec(R), X = Spec(S)$ and $R \to S$ is finite étale of
rank $n$. Then the vector bundle $M$ corresponds to a locally free $S$-module which we still
denote by $M$.

There is an evident map
\[ \alpha_M: M \to \nu_{S/R}(M),\,\, m \mapsto (m \otimes m \otimes
\dots \otimes m) \otimes 1. \]
To be clear, this is not a homomorphism.
I claim that it defines a norm law. We first need
to show that it is a polynomial law. 
Hence if $R_1 \to R_2$ is a
homomorphism of $R$-algebras, we need to show that the following square
commutes:
\begin{equation*}
\begin{CD}
M \otimes_R R_1 @>{\alpha_{M \otimes_R R_1}}>> \nu_{S \otimes_R R_1/R_1}(M) \\
@VVV                     @VVV                  \\
M \otimes_R R_2 @>{\alpha_{M \otimes_R R_2}}>> \nu_{S \otimes_R R_2/R_2}(M).
\end{CD}
\end{equation*}
Here the left vertical map is the induced one. We have $\Sym^n_{R'}(M \otimes_R R')
\iso \Sym^n_R(M) \otimes_R R'$ \cite[Korollar 2.3.2]{krsnik2006multiplikative} and
consequently $\nu_{S \otimes_R R'/R'}(M) \iso \nu_{S/R}(M) \otimes_R R'$. The
right vertical map is the one induced by this isomorphism.
Checking
commutativity is then routine. The proof that this is a norm law boils down to
the claim that for $s \in S$ we have $\nu_{S/R}(s \otimes s \otimes \dots
\otimes s) = n_{S/R}(s) \in R$. This
follows from \cite[Korollar 4.1.2]{krsnik2006multiplikative} and
\cite[(N1)]{ferrand1998foncteur}.

By universality of the norm law $M \to N_{S/R}(M)$ there exists a unique
$R$-linear map $\beta_M: N_{S/R}(M) \to \nu_{S/R}(M)$ such that the composite $M \to
N_{S/R}(M) \to \nu_{S/R}(M)$ is $\alpha_M$. We claim that $\beta_M$ is an
isomorphism. To see this, we may perform a faithfully flat base change and
assume that $S \iso R^d$. But then $M \iso \prod_{i=1}^d M_i$ and $N(M) \iso
\bigotimes_i M_i \iso \nu(M)$ \cite[Lemme 3.2.4]{ferrand1998foncteur}
\cite[Satz 4.3.2]{krsnik2006multiplikative} (this is where we need $M$ locally
free). Moreover $n_M: M \to N(M)$ is given by $(m_1, \dots, m_d) \mapsto (m_1
\otimes m_2 \otimes \dots m_d)$, c.f. \cite[Exemple
3.2.2.c)]{ferrand1998foncteur}. It follows from uniqueness of $\beta_M$ that
$\beta_M$ is indeed the canonical isomorphism. This concludes the proof.
\end{proof}

In the last paragraph of the above proof we have also established the following
result:

\begin{prop} \label{prop:norm-fold}
Let $f: X \coprod X \coprod \dots \coprod X \to X$ be the fold map of an
$n$-fold coproduct. Then $N_f: Vect(X \coprod \dots \coprod X) \iso Vect(X)^n
\to Vect(X)$ is given by $(E_1, \dots, E_n) \mapsto E_1 \otimes \dots \otimes
E_n$.
\end{prop}

If $f: X \to Y$ is finite étale and $E \in Vect(X)$, then $\nu_f(E) \in Vect(Y)$,
as follows for example from \cite[Satz 4.3.2]{krsnik2006multiplikative}. The
functor $\nu_f = N_f: Vect(X) \to Vect(Y)$ is symmetric monoidal \cite[Korollar
4.3.4]{krsnik2006multiplikative}. Since $N_f(\mathcal{O}_X) = \mathcal{O}_Y$ it
follows easily that the functor preserves bilinear bundles, cf. also \cite[Korollar
4.2.7]{krsnik2006multiplikative}. We have thus found $N_f: Iso(Vect(X)) \to
Iso(Vect(Y))$ and $N_f: Iso(Bil(X)) \to Iso(Bil(Y))$, and these are homomorphisms
of \emph{multiplicative} monoids.

\section{$GW$ as a Tambara functor}
\label{sec:tambara}

Recall that if $f: X \to Y$ is any morphism of schemes, then the pullback $f^*:
QCoh(Y) \to QCoh(X)$ induces $f^*: Iso(Vect(Y)) \to Iso(Vect(X))$ and $f^*:
Iso(Bil(Y)) \to Iso(Bil(X))$, and these are homomorphisms of semi-rings, i.e.
respect both the multiplicative and additive structure. Suppose that $f$ is
finite étale. Then as we have seen in the previous section, we can construct $N_f:
Iso(Bil(X)) \to Iso(Bil(Y))$, which is a homomorphism of multiplicative monoids,
and $tr_f: Iso(Bil(X)) \to Iso(Bil(Y))$, which is a homomorphism of additive
monoids. In this section we will explain the compatibilities between these three
operations. It turns out that Tambara
\cite{tambara1993multiplicative} has studied precisely this kind of situation.
As a consequence, we will be able to deduce that $N_f: Iso(Bil(X)) \to
Iso(Bil(Y))$ extends in a canonical way to $N_f: K(Bil(X))^\oplus \to
K(Bil(Y))^\oplus$, and that this extension has many desirable properties, such
as a base change formula. This extension has been produced before in a different way by
M. Rost \cite{rost2003multiplicative}.

We write $\Fet$ for the category of all schemes, with morphisms the finite étale
morphisms. For a scheme $S$, we let $\Fet/S$ denote the usual slice category.
Recall that any morphism between schemes which are finite étale over $S$ is
automatically finite étale, so $\Fet/S$ is also the category of finite étale
schemes over $S$, with all morphisms of schemes allowed.

If $f: X \to Y \in \Fet/S$ is a morphism, then we get as usual a functor $f^*:
\Fet/Y \to \Fet/X$. This functor has a right adjoint $f_*$ which is in fact
isomorphic to Weil restriction along $f$. To see this, it suffices to show that
if $T \in \Fet/X$ then the Weil restriction $R_{X/Y}(T) \in Sch/Y$ is finite étale.
This is clear from
infinitesimal lifting criteria; see \cite[Proposition 7.5.5]{bosch2012neron}.

The following definitions are modeled in an evident way on
\cite[Definition 1.4.1]{mazur2013structure}.

\begin{defnn}
Given morphisms $A \xrightarrow{q} X \xrightarrow{f} Y$ in $\Fet/S$, we can
build the following commutative diagram (in $\Fet/Y$ or $\Fet/S$) called the \emph{exponential diagram generated by $A \to X \to Y$}:
\begin{equation*}
\begin{CD}
X @<q<< A @<e<< f^* f_* A \\
@VfVV   @.      @VpVV     \\
Y  @>{\iso}>>   f_* X  @<{f_*q}<<   f_* A.
\end{CD}
\end{equation*}
Here $f_*$ denotes the functor $f_*: \Fet/X \to \Fet/Y$. Being a right adjoint
it preserves final objects, whence the claimed isomorphism. The left adjoint is
$f^*: \Fet/Y \to \Fet/X$.
The map $e: f^*f_* A \to A$ is the counit of this adjunction.
The functor $f^*$ also has a left adjoint $f_\#$, and $p: f_\# f^* f_* A \to f_*
A$ is the counit of this further adjunction.

As in the exponential diagram above, we usually will not distinguish between
$f_\# T$ and $T$. Indeed $f_\# T$ is just ``$T$ viewed as a $Y$-scheme in the
canonical way'', so this should not cause confusion.
\end{defnn}

\begin{defnn}\label{def:tambara}
A \emph{Tambara functor $F$ over $S$} consists of the following data:
for each $X \in \Fet/S$ a
semi-ring $F(X)$, together with for each $f: X \to Y \in \Fet/S$ three maps
$f^*: F(Y) \to F(X), tr_f: F(X) \to F(Y)$ and $N_f: F(X) \to F(Y)$. This data is
required to satisfy the following conditions:
\begin{enumerate}
\item $F(X \coprod Y) \iso F(X) \times F(Y)$, via the canonical map
\item $f^*$ is a homomorphism, $tr_f$ is a homomorphism of additive monoids,
  and $N_f$ is a homomorphism of multiplicative monoids
\item $f^*, tr_f, N_f$ are functorial in $f$ (i.e. $(fg)^* = g^* f^*$, $\id^* =
\id$, and so on)
\item Transfer and norm commute with base change, in the sense that given a
  cartesian square in $\Fet/S$
\begin{equation*}
\begin{CD}
X @>f>> Y \\
@VgVV @VhVV \\
Z @>k>> W, 
\end{CD}
\end{equation*}
  the following square commutes
\begin{equation*}
\begin{CD}
F(X) @>{tr_f}>> F(Y) \\
@Ag^*AA      @Ah^*AA \\
F(Z) @>{tr_k}>> F(W),
\end{CD}
\end{equation*}
  and similarly with $N$ in place of $tr$.
\item Given morphisms $Z \xrightarrow{q} X \xrightarrow{f} Y$ (in $\Fet/S$), the following
  diagram induced by the associated exponential diagram commutes:
\begin{equation*}
\begin{CD}
F(X) @<tr_q<< F(Z) @>e^*>> F(f^* f_* Z) \\
@VN_fVV   @.              @VN_pVV     \\
F(Y)  @>{\iso}>>   F(f_*X)  @<{tr_{f_*q}}<<   F(f_* Z).
\end{CD}
\end{equation*}
\end{enumerate}
\end{defnn}

\begin{rmk} \label{rmk:tambara-cons}
Conditions 2, 3 and 4 above imply that if $f: X^{\coprod n} \to X$ is the fold map,
then $tr_f: F(X^{\coprod n}) \iso F(X)^n \to F(X)$ is just addition in $F(X)$,
and similarly $N_f$ is multiplication.
\end{rmk}
\begin{rmk} \label{rmk:tambara-defn}
Condition 5 above should be seen as a generalized \emph{distributivity law}.
Indeed if all finite étale morphisms occurring are fold maps (i.e. are split), then this condition
precisely expresses that multiplication distributes over addition in the
semi-ring $F(X)$. As a consequence, we see that a Tambara functor can
alternatively be
defined as a presheaf of sets (not semi-rings) with norms and transfers,
satisfying conditions 1, 3, 4, 5; then there is a unique structure of a
semi-ring on each $F(X)$ such that 2 is satisfied.
We will not use this observation.
\end{rmk}

\begin{prop}
The assignments $X \mapsto Iso(Vect(X))$ and $X \mapsto Iso(Bil(X))$ define
Tambara functors on $S$.
\end{prop}
\begin{proof}
We need to verify the axioms. Condition 1 is obvious, condition 2 is obvious for
$f^*$. It holds for $tr_f$ because $f_*$ preserves direct sums, and it holds for $N_f$
because the norm functor is symmetric monoidal, see for example \cite[Satz
4.4.4]{krsnik2006multiplikative}. Alternatively, the argument of Remark
\ref{rmk:tambara-defn} reduces this to Proposition \ref{prop:norm-fold}.
Condition 3 for $f^*$ is well-known, for $tr_f$ it follows from functoriality of
$f_*$, and for $N_f$ it follows from functoriality of the norm \cite[Proposition
3.2.5 b) and c)]{ferrand1998foncteur}. Similarly condition 4 follows from the
base change formula for $f_*$ (well-known) and for the norm \cite[(N2)]{ferrand1998foncteur}.

We thus need to establish condition 5. We do this in some detail; this style of
argument can also be used to make more precise our sketches for conditions 1-4.
Given morphisms $Z \xrightarrow{q} X \xrightarrow{f} Y$ in $\Fet/S$ and $a
\in Iso(Vect(Z))$ (respectively $\bar{a} \in Iso(Bil(Z))$) we need to show that
$N_f tr_q(a) = tr_{f_* q} N_p e^*(a)$, and similarly for $\bar{a}$. In order to
do this, for $E \in Vect(Z)$, we shall (1) exhibit a natural morphism $N_f q_* E \to
(f_* q)_* N_p E$. Natural here means natural in $E$ and also in $Y$, i.e.
compatible with base change. Then we shall (2) verify that this morphism is an
isomorphism, and that if $E \in Bil(Z)$, then the morphism respects the bilinear
structure. Note that problem (1) is Zariski-local on $X$, and problem (2) is
even local in the faithfully flat topology. Consequently to solve (1) we may
assume that $Y$ is affine, say $Y = Spec(A)$, and for (2) we may even assume
that $X = Spec(A^n)$.

\paragraph{Step 1.} We know that
Weil restriction is isomorphic to the norm construction \cite[Proposition
6.2.2]{ferrand1998foncteur}. We are thus given the following diagram of
commutative rings (with all maps finite étale homomorphisms)
\begin{equation*}
\begin{CD}
B @>q>> C @>e>> N_f C \otimes_A B \\
@AfAA  @.        @ApAA            \\
A @=     N_f B   @>r>>    N_fC
\end{CD}
\end{equation*}
together with a locally free module $M$ on $C$, and we need to exhibit a
functorial morphism of $A$-modules
  \[ N_f M \to N_p(M \otimes_C (N_fC \otimes_A B)). \]
Recall that on the level of modules, the pushforward operation just coincides
with forgetting some of the module structure; this is why there are no transfer
operators in the above formula.

We shall construct a norm law $B\Mod \ni M \to N_p(M') \in A\Mod$,
where $M' := M \otimes_C (N_fC \otimes_A
B)$. By universality of the norm law $M \to N_fM$, this induces a
\emph{homomorphism} $N_f M \to N_p M'$ as desired.

To do this,
consider the composite $\phi: M \to M' \to N_p(M')$, where $M \to M'$ is $m
\mapsto m \otimes 1$ and $M' \to N_p(M')$ is the norm map. This defines a
polynomial law over $A$ because the entire diagram is functorial under base
change in $A$, which follows from condition 4 (which we have already
established).
To see that this defines a norm
law (see Definition \ref{defn:norm-law}) we need to check
that for $b \in B$ and $m \in M$ we have
\[ \phi(bm) = n_f(b) \phi(m). \]
Here for any $B$-module $L$, we write $n_f: L \to N_f L$ for the universal norm
law. By definition we have $\phi(m) = n_p(m \otimes 1)$ and hence $\phi(bm) =
n_p(e(q(b))) \cdot \phi(m)$, since $n_p$ is a norm law. We thus need to show that
\[ n_p(e(q(b))) = r(n_f(b)) .\] Since
all our modules are locally free, hence flat, we may check this after any
base change along an injective ring homomorphism, e.g. a faithfully flat one.

We may thus assume that $B = A^d$ and hence that $C = \prod_{i=1}^d C_i.$
Then $N_fC = \bigotimes_{i=1}^d
C_i.$ The canonical map $e: C \to N_f C \otimes_A B = (N_f C)^d$ is given by
\[ (c_1, \dots, c_d) \mapsto (c_1 \otimes 1 \otimes \dots, 1 \otimes c_2 \otimes
  1 \otimes \dots, \dots, 1 \otimes 1 \otimes \dots \otimes 1 \otimes c_d), \]
as one checks by verifying the universal property of a unit of adjunction.
The norm maps are $n_f: B = A^d \to A, (a_1, \dots, a_d) \mapsto a_1 \dots a_d$
and similarly for $n_p$. Then $n_peq = rn_f: B \to N_f C$ follows by direct
computation.

\paragraph{Step 2.}
We remain in the situation above, i.e. we have $A \to B \to C$ with $B = A^d$
and $C = \prod_i C_i$. Moreover we have a $B$-module $M$ and canonical morphism
$N_f M \to N_p M'$. We need to check that this is an isomorphism, respecting
bilinear structures if present.

We first check that the morphism is an
isomorphism. We have $M = \prod_{i=1}^d M_i$, where $M_i$ is a
$C_i$-module. Then $N_f M = \bigotimes_{i=1}^d M_i$, where the tensor product is
over $A$. In contrast, $M' = \prod_{i=1}^d (M_i \otimes_{C_i} N_f C)$, and $N_p
M'$ is the tensor product of these terms over $N_f C$. Since $M_1
\otimes_{C_1} N_f C = M_1 \otimes C_2 \otimes \dots \otimes C_d$ and so on, it
is easy to see by direct computation that $N_f M \to N_p M'$ is an isomorphism.

Similarly, giving a bilinear form $\psi: M \otimes_B M \to B$ is equivalent to
giving $\psi_i: M_i \otimes M_i \to A$. The induced bilinear form on $N_f M =
\bigotimes_i M_i$ is $N_f(\psi)(m_1 \otimes \dots \otimes m_d) = \prod_i
\psi_i(m_i)$. Arguing similarly for $M'$, it is easy to check by direct
computation that $N_f M \to N_p M'$ is compatible with the bilinear forms. This
concludes the proof.
\end{proof}

\begin{rmk}
We can formally invert the sum operation in $Iso(Vect(X))$
and then obtain the Grothendieck ring $K(Vect(X))^\oplus$.
It is a priori not at all clear that the norm map $N_f: Iso(Vect(X))
\to Iso(Vect(Y))$ induces a map $K(Vect(X))^\oplus \to K(Vect(Y))^\oplus$. The main
point of \cite{rost2003multiplicative} is that this indeed works, and the proof
is by showing that the norm maps are \emph{polynomial} (in a sense that is a
priori stronger than the definition we have used so far) and then showing that
polynomial maps descend to Grothendieck groups.

It is also possible to deduce this fact from our proposition. Indeed, any
Tambara functor may be ``additively completed'' (i.e. one may pass to the
Grothendieck ring) \cite[Theorem 6.1]{tambara1993multiplicative}.

The same discussion can be repeated with $Bil(X)$ in place of $Vect(X)$.
\end{rmk}

Using either of the above mentioned results, we obtain the following.

\begin{corr} \label{corr:GW-tambara}
The assignments $X \mapsto K(Vect(X))^\oplus$ and $X \mapsto K(Bil(X))^\oplus$ define
Tambara functors on $S$.
\end{corr}

\section{The $GW$-module structure on $GW^\times$}
\label{sec:GW-module-structure}
In this section we begin in earnest the program sketched in the introduction:
using the multiplicative transfers studied in the previous two sections, we turn
the group of units in $GW(k)^\times$ into a module over $GW(k)$, where $k$ is a
field.

We consider the Tambara functor
$K(Bil(\bullet)) = GW(\bullet)$ on $\Fet/k$. If $A/k$ is a finite étale algebra,
we write $N_{A/k}$ and $tr_{A/k}$ for the multiplicative and additive transfer,
and $x \mapsto x|_A$ for the restriction. We put $tr(A) := tr_{A/k}(1)$.

Recall the dimension homomorphism $dim: GW(A) \to \ZZ^{d}$, where $d$ is the
number of connected components of $Spec(A)$. Its kernel is called the
fundamental ideal and denoted $I(A)$. Note that $I(A \times B) = I(A) \times
I(B)$. Since $GW(A)$ is a ring it has a subset
of units $GW^\times(A)$. This is a group where the operation is multiplication
in the Grothendieck--Witt ring. We put for $n \ge 0$
\[ F_n GW^\times(A) = \left\{ x \in GW^\times(A) \,|\, x \equiv 1 \pmod{I^n(A)} \right\}. \]
In other words, $F_n GW^\times(A) = (1 + I^n(A)) \cap GW^\times(A)$. We will
make good use of the map
\begin{equation} \label{eq:alphan}
  \alpha_n: F_n GW^\times(A) \to I^n(A)/I^{2n}(A),\, x \mapsto x-1.
\end{equation}
Note that it is a homomorphism,
where we use the multiplicative group structure on the left and the additive
structure on the right. In fact $I^n(A)/I^{2n}(A)$ is the largest quotient
$I^n(A)/J$ of
$I^n(A)$ such that the map $(1 + I^n(A), \times) \to (I^n(A)/J, +), x \mapsto
x-1$ is a homomorphism of monoids.

Since we shall use it all the time, let us make explicit the following
well-known fact.

\begin{lemm}[Arason \cite{arason1975cohomologische}, Satz 3.3]
    \label{lemm:transfer-I}
Let $A/k$ be a finite étale algebra. Then $tr_{A/k}(I^n(A)) \subset I^n(k)$.
\end{lemm}


For the purpose of this section, we will always view $I(k)$ as an
ideal of $GW(k)$. We thus put $I^0(k) := GW(k)$, and not $I^0(k) = W(k)$ as may
be more customary.

The following well-known result is very useful for computations.
\begin{lemm}\label{lemm:odd-degree-base-change}
Let $l/k$ be an algebraic field extension of odd degree. In other words $l/k$ is
algebraic, and if $l/l_0/k$
is a subextension with $l_0/k$ finite, then $[l_0:k]$ is odd. For $0 \le n \le
\infty$, the restriction map
\[ I^n(k)/I^m(k) \to I^n(l)/I^m(l) \]
is injective. Here we put $I^\infty(k) := 0$.
\end{lemm}
\begin{proof}
In the proof we shall use transfers along finite extensions which are not
separable in general; these are also known as Scharlau transfers. The only
difference is that there need not be a unique transfer, but rather there will be
a family of transfers differing by multiplication by a one-dimensional form. The
choice will not matter to us, since we use only the most basic properties. The
most subtle one is Lemma \ref{lemm:transfer-I}, which holds
in this generality (and is in fact stated in this generality by Arason). In this
text we shall only ever apply Lemma \ref{lemm:odd-degree-base-change} to
separable extensions anyway.

By continuity (see Corollary \ref{corr:app-cont} in Appendix
\ref{sec:app-cont}), we may assume that $l/k$ is a finite extension.

Let us first show that $I^n(k)/I^{n+1}(k) \to I^n(l)/I^{n+1}(l)$ is injective.
Since transfer preserves $I^n$ by Lemma \ref{lemm:transfer-I}, we get a well-defined map $tr:
I^n(l)/I^{n+1}(l) \to I^n(k)/I^{n+1}(k)$. It suffices to show that the composite
$\alpha: I^n(k)/I^{n+1}(k) \to I^n(l)/I^{n+1}(l) \to
I^n(k)/I^{n+1}(k)$ is injective. By the projection formula (see
\cite[Theorem VII.1.3]{lam-quadratic-forms} for the inseparable case), $\alpha$ is given by
multiplication by an element $t \in GW(k)$ of dimension $[l:k]$. In other words
$t \equiv [l:k] \pmod I$ and consequently $\alpha$
is given by multiplication by $[l:k]$. Since $2I(k) \subset I^2(k)$, i.e.
$I^n(k)/I^{n+1}(k)$ is an $\FF_2$-vector space, multiplication by the odd
integer $[l:k]$ is injective.

If $0 \le n \le m < \infty$ then $I^n(k)/I^m(k) \to I^n(l)/I^m(l)$ is a morphism
of finitely filtered abelian groups which is injective on the subquotients (by what we
have just shown), hence injective. The case $m=\infty$ follows from the fact
that $\cap_m I^m(k) = 0$ \cite[Corollary 4.5.7]{scharlau2012quadratic}, or directly
from Springer's theorem \cite[Corollary 2.5.4]{scharlau2012quadratic}.
\end{proof}

By a
\emph{degree 2 extension} $A$ of a field $k$ of characteristic different from
2 we shall always mean a degree 2
étale extension; this is either a quadratic field extension
of $k$, or the extension $A = k \times k$. Such an extension has a canonical
automorphism denoted by $x \mapsto \bar{x}$. If $A/k$ is a quadratic extension
then $x \mapsto \bar{x}$ is the non-trivial Galois automorphism. If $A = k
\times k$ then we put $\overline{(x,y)} := (y, x)$. Then for any $A/k$ of
degree 2, the formulas
$tr_{A/k}(x) = x + \bar{x}$ and $N_{A/k}(x) = x\bar{x}$ are correct for all $x
\in A$, and $k \hookrightarrow{A}$ consists precisely of the invariants of $x
\mapsto \bar{x}$. We then have the following slight extension of a result of
Wittkop we shall use extensively.

\begin{lemm}[Wittkop \cite{wittkop2006multiplikative}]
   \label{lemm:wittkop-formula}
Let $char(k) \ne 2$, $A/k$ of degree 2 and $x, y \in GW(A)$. Then
\[ N_{A/k}(x + y) = N_{A/k}(x) + N_{A/k}(y) + tr_{A/k}(x\bar{y}). \]
\end{lemm}
\begin{proof}
If $A/k$ is quadratic, then this is \cite[Satz
2.5(ii)]{wittkop2006multiplikative}. If $A = k \times k$ then $x = (x_1, x_2), y
= (y_1, y_2)$ and by Proposition \ref{prop:norm-fold} (and its additive
analogue, cf. Remark \ref{rmk:tambara-cons}) we have
\[ N(x + y) = (x_1 + y_1)(x_2 + y_2) = x_1x_2 + y_1y_2 + (x_1 y_2 + x_2 y_1)
            = N(x) + N(y) + tr(x\bar{y}). \]
\end{proof}

The following is a very basic result. Its proof illustrates nicely how to use
Lemmas \ref{lemm:odd-degree-base-change} and \ref{lemm:wittkop-formula}.
We give many details because we shall re-use the technique several times. Note
in particular how testing equalities between elements constructed using norms
and transfers is reduced to the case of degree 2 extensions.

\begin{prop} \label{prop:norm-basics}
Let $A/k$ be a finite étale algebra.
\begin{enumerate}[(i)]
\item If $A/k$ is of degree $p$, we have $N_{A/k}(I^n(A)) \subset I^{np}(k)$.
\item We have $N_{A/k}(F_n GW^\times(A)) \subset F_n GW^\times(k)$.
\item Recall the map $\alpha_n$ from equation
    \eqref{eq:alphan}. In the situation of (ii), moreover the following diagram commutes:
\begin{equation*}
\begin{CD}
F_n GW^\times(A) @>\alpha_n>> I^n(A)/I^{2n}(A)   \\
@VN_{A/k}VV                @Vtr_{A/k}VV  \\
F_n GW^\times(k) @>\alpha_n>> I^n(k)/I^{2n}(k).
\end{CD}
\end{equation*}
\end{enumerate}
\end{prop}
\begin{proof}
Let $l/k$ be a finite Galois extension such that all residue fields of $A$
embed in $l$, $H$ a 2-Sylow subgroup of $G = Gal(l/k)$
and $k' = {l}^H/k$ the associated field extension. Let $A' := A \otimes_k k'$ denote the
scalar extension. By Corollary \ref{corr:GW-tambara} we have a commutative diagram
\begin{equation*}
\begin{CD}
GW(A) @>>> GW(A') \\
@V{N_{A/k}}VV   @V{N_{A'/k'}}VV \\
GW(k) @>>> GW(k').
\end{CD}
\end{equation*}

For (i), we wish to show that the composite $I^n(A) \to GW(A)
\xrightarrow{N_{A/k}} GW(k) \to
GW(k)/I^{np}(k)$ is zero. We stress that it is not a homomorphism; this will
not matter.
Since $[k':k] = |G/H|$ is odd, by Lemma
\ref{lemm:odd-degree-base-change} we know that $GW(k)/I^{np}(k) \to
GW(k')/I^{np}(k')$ is injective. Hence consulting the commutative diagram, we
find that we may assume (replacing $k$ by $k'$) that $G$ is a 2-group.

I claim that in this case, for any subextension $k \subset l' \subset l$ there exist
intermediate extensions $k = k_0 \subset k_1 \subset \dots \subset k_p = l'$ such that
$[k_{i+1}:k_i] = 2$ for all $i$.
To see this, we may assume by induction that
$k \subset l'$ has no intermediate extensions, i.e. $Gal(l/l') \subset Gal(l/k)$
is a proper maximal subgroup. Any such subgroup is of index $2$
\cite[Theorem 4.6]{rotman2012introduction}, so the claim is proved.

Now $A = l_1 \times \dots \times l_r$, where by construction each $l_i$ embeds
into $l$.
Suppose that we can prove (i) for each of the $l_i/k$. Then if $x = (x_1, \dots,
x_r) \in I^n(A)$ we get
\[ N_{A/k}(x) = \prod_i N_{l_i/k}(x_i) \in \prod_i I^{[l_i:k]n}(k) \subset
I^{np}(k), \]
since $p = \sum_i [l_i:k] = [A:k]$.
We thus need only prove the result for $A=l$ a field. In this case we use the claim to factor $k
\to A=l$ as $k \to k_1 \to \dots \to k_q = l$, with $[k_{i+1}:k_i] = 2$. Using transitivity of the
norm we reduce to the case $[l:k] = 2$.

If $x, y \in I^n(l)$ with $N_{l/k}(x),
N_{l/k}(y) \in I^{2n}(k)$ then $N_{l/k}(x + y) = N_{l/k}(x) + N_{l/k}(y) +
tr_{l/k}(x\bar{y}) \in I^{2n}(k)$, by Lemmas \ref{lemm:wittkop-formula} and
\ref{lemm:transfer-I}.
Note that the conjugation $y \mapsto \bar{y}$ preserves $I^n(l)$.
Moreover in this situation we have $N_{l/k}(-x) = N_{l/k}(-1)N_{l/k}(x)
\in I^{2n}(k)$. It is thus enough to show that $N_{l/k}(x) \in I^{2n}(k)$ for an
additive generating set of elements $x$ of $I^{n}(l)$. This follows from
\cite[Lemma 1.56 and Satz 2.16 (ii),(iv)]{wittkop2006multiplikative}.  
We have thus proved (i).

The proof of (ii) proceeds similarly, using the
composite $F_n GW^\times(A) \xrightarrow{N_{A/k}} GW^\times(k) \xrightarrow{x
\mapsto x-1} GW(k)/I^{2n}(k)$. We may assume that there is a Galois extension
$l/k$ with $Gal(l/k)$ a $2$-group such that each residue field of $A$ embeds
into $l$. We factor $k \to A$ as $k \to k^r \to \prod_i l_i$; this allows us to
reduce to $A$ a field, which using the claim we may assume is of degree $2$ over
$k$, or $A = k^r$. But $N_{k^r/k}(x_1, \dots, x_r) = \prod_i
x_i = N_{k^2/k}(x_1, N_{k^2/k}(\dots N_{k^2/k}(x_{r-1}, x_r) \dots))$, so the
case $A = k^r$ reduces to $A = k^2$. In other words we have reduced to the case
of a degree $2$ extension.

Thus we
need to show that if $A/k$ is of degree $2$ and $1 + x \in F_n GW^\times(A)$,
then $N_{A/k}(1 + x) \in F_n GW^\times(k)$. Since the norm is multiplicative it is
clear that $N(1+x)$ is invertible, so it suffices to show that $N_{A/k}(1+x) \in 1 +
I^n(k)$. But $N_{A/k}(1+x) = 1 + tr_{A/k}(x) + N_{A/k}(x)$ by Lemma \ref{lemm:wittkop-formula}
again, and $tr_{A/k}(x), N_{A/k}(x) \in
I^n(k)$ by Lemma \ref{lemm:transfer-I} and part (i).

To prove part (iii), we may again assume that $A/k$ is of degree 2.
Let $1 + x \in F_n GW^\times(A)$.
Then $N_{A/k}(1+x) = 1 + N_{A/k}(x) + tr_{A/k}(x)$ by Lemma \ref{lemm:wittkop-formula}
once more. Since $N_{A/k}(x) \in I^{2n}(k)$
by (i) we have
\[ \alpha_n(N_{A/k}(1+x)) = [N_{A/k}(x) + tr_{A/k}(x)] = [tr_{A/k}(x)] =
tr_{A/k}(\alpha_n(1+x)), \] which concludes the proof.
\end{proof}

\begin{rmk}
The above proof can be made more uniform by using \emph{locally constant
integers}. Namely, if $p$ is a function on $Spec(A) = \coprod_i
Spec(l_i)$ with
values in $\NN$ (i.e. a non-negative locally constant integer on $Spec(A)$),
then we put $I^p(A) = \prod_i I^{p(Spec(l_i))}(l_i) \subset
GW(A)$. Now suppose that $f: A_1 \to A_2$ is a finite étale morphism and $p$ is a
locally constant integer on $Spec(A_2)$. We define a locally constant integer
$tr_{A_2/A_1}(p)$ on $Spec(A_1)$ by
\[ tr_{A_2/A_1}(p)(x) = \sum_{y \in Spec(f)^{-1}(\{x\})} [k(y):k(x)] p(y). \]
Note that if $A/k$ is of degree $p$ and we view $n \in \ZZ$ as a constant
function on $Spec(A)$, then $tr_{A/k}(n) = np$.

Now suppose that $A_1 \to A_2$ is a finite étale morphism of finite étale
$k$-algebras, and $n$ is a locally constant integer on $Spec(A_2)$. One may show
that then
\[ N_{A_2/A_1}(I^n(A_2)) \subset I^{tr_{A_2/A_1}(n)}(A_1). \]
This statement includes statement (i) above as a special case. It is not
difficult to use the same ideas as in the proof of (ii) to reduce this more
general statement to the case of a degree $2$ extension.  
We chose to give an ad hoc argument to avoid the complication of locally
constant integers. Similar remarks apply at various points in the sequel.
\end{rmk}

\begin{rmk} \label{rmk:degree-zero-Ntr}
For $n = 0$ the statement (iii) of the proposition is not
useful. Instead, I claim that the
following diagram commutes:
\begin{equation*}
\begin{CD}
GW^\times(A) @>{dim}>> \{\pm 1\}^{Spec(A)}   \\
@VN_{A/k}VV                @VNVV  \\
GW^\times(k) @>{dim}>> \{\pm 1\},
\end{CD}
\end{equation*}
where $N$ on the right hand side is defined as follows. Suppose $A = l_1 \times
\dots \times l_r$, with each $l_i$ a field, and $f \in \{\pm 1\}^{Spec(A)}$.
Then $N(f) = \prod_i f(Spec(l_i))^{[l_i:k]}$. Note that $Spec(A) \mapsto
\{\pm 1\}^{Spec(A)}$ is a presheaf, $\{\pm 1\}^{Spec(k)} = \{\pm 1\}$ for any
field $k$, and that the operation $N$ defined above satisfies the base change
formula. In particular, for any field extension $l/k$, the following diagram
commutes:
\begin{equation*}
\begin{CD}
\{\pm 1\}^{Spec(A)} @>>> \{\pm 1\}^{Spec(A \otimes_k l)}   \\
@VNVV                @VNVV  \\
\{\pm 1\}^{Spec(k)} @= \{\pm 1\}^{Spec(l)}.
\end{CD}
\end{equation*}
Since $GW$ also satisfies the base change formula, in order to prove the claim
we may thus assume that $k$ is algebraically
closed. In this case $dim: GW(A) \to \ZZ^{Spec(A)}$ is an isomorphism, and the diagram
commutes by Proposition \ref{prop:norm-fold}.
\end{rmk}

The following observation will allow us to turn the norm maps into a $GW$-module
structure.

\begin{lemm} \label{lemm:main-observation}
Let $A/k, B/k$ be finite étale algebras and $x \in GW(k)$. If $tr(A) = tr(B) \in
GW(k)$ then $N_{A/k}(x|_A) = N_{B/k}(x|_B) \in GW(k)$.
\end{lemm}
\begin{proof}
Since $tr(A) = tr(B)$ we must have $[A:k] = [B:k]$. It follows from Remark
\ref{rmk:degree-zero-Ntr} that $dim(N_{A/k}(x|_A)) = dim(x)^{[A:k]} =
dim(N_{B/k}(x|_B))$. It is thus enough to show that $[N_{A/k}(x|_A)] =
[N_{B/k}(x|_B)] \in W(k)$. For this we use Serre's splitting principle. For $l/k$ some field
extension and $E \in Et_n(l)$ an étale algebra of degree $n := [A:k]$, let
$\phi(E) = [N_{E/l}(x|_E)] \in W(l)$. This defines an invariant in the sense of
\cite[Definition 1.1]{garibaldi2003cohomological}, by the base change formula.
It follows from \cite[Theorem 29.2]{garibaldi2003cohomological}
that there exist $x_0, \dots, x_n \in W(k)$ (depending only on $x$,
not on $E$) such that
\[ \phi(E) = x_0 + x_1 [\lambda^1(tr(E))] + \dots + x_n[\lambda^n(tr(E))]. \]
The claim follows.
\end{proof}

\begin{rmk}
If $E/k$ is a quadratic extension, then it follows from \cite[Satz
2.10]{wittkop2006multiplikative} that
\[ [N_{E/k}(x|_E)] = [x^2 - 2\lambda^2(x)] + [\lambda^2(x)] \cdot [tr(E)] \in W(k), \]
where $\lambda^2(x)$ refers to the canonical $\lambda$-ring structure on
$GW(k)$.
In other words, in the last step of the proof, if $n=2$, one may take $x_0 =
[x^2 - 2\lambda^2(x)]$ and $x_1 = [\lambda^2(x)]$ (and $x_2 = 0$).
\end{rmk}

If $l = k(\sqrt{a})$ then $tr(l) = \langle 2 \rangle + \langle 2a \rangle$
\cite[Lemma 2.3 (ii)]{wittkop2006multiplikative}. It follows easily that any
element $y \in GW(k)$ may be written as $tr(A) - tr(B)$ for $A/k, B/k$ finite
étale algebras. Then for $x \in GW^\times(k)$ we put $x^y :=
N_{A/k}(x|_A)/N_{B/k}(x|_B)$. This division is well-defined because the norm preserves
units.

\begin{prop} \label{prop:GW-module-existence-basics}
\begin{enumerate}[(i)]
\item The element $x^y \in GW^\times(k)$ is well-defined, independent of the
  choice of representation $y = tr(A) - tr(B)$.
\item The pairing
  \[ GW(k) \times GW^\times(k) \to GW^\times(k), (y, x) \mapsto x^y \] turns
  $GW^\times(k)$ into a $GW(k)$-module.
\item Each of the subgroups $F_n GW^\times(k) \subset GW^\times(k)$ is a
  $GW(k)$-submodule.
\item The $GW(k)$-module $GW^\times(k)$ satisfies the projection formulas: For
  $A/k$ finite étale, we have
  \[ N_{A/k}(x_1^{y_2|_A}) = N_{A/k}(x_1)^{y_2} \quad\text{ for } 
            x_1 \in GW^\times(A), y_2 \in GW(k) \] and
  \[ N_{A/k}((x_2|_A)^{y_1}) = x_2^{tr_{A/k}(y_1)} \quad\text{ for }
             y_1 \in GW(A),x_2 \in GW^\times(k) . \]
\end{enumerate}
\end{prop}
\begin{proof}
Let us note first that if $A, B$ are finite étale $k$-algebras and $(x, y) \in
GW(A \times B) = GW(A) \times GW(B)$, then $N_{A \times B/k}(x, y) = N_{A/k}(x)
N_{B/k}(y)$. This follows from transitivity of the norm and Proposition
\ref{prop:norm-fold} by factoring $k \to A \times B$ as $k \to k \times k \to A
\times B$ where the first map is the diagonal.

(i) If $tr(A) - tr(B) = tr(A') - tr(B')$ then $tr(A \times B') = tr(A) + tr(B')
= tr(A') + tr(B) = tr(A' \times B)$ and consequently
\[ N_{A/k}(x|_A)N_{B'/k}(x|_{B'}) = N_{A \times B'/k}(x|_{A \times B'})
   = N_{A' \times B/k}(x|_{A' \times B}) = N_{A'/k}(x|_{A'})N_{B/k}(x|_{B}), \]
where for the middle equality we have used Lemma \ref{lemm:main-observation},
and for the outer equalities we use the first paragraph of this proof.
The claim follows upon division by $N_{B'/k}(x|_{B'}) N_{B/k}(x|_{B})$.

(ii) Write $y = tr(A) - tr(B)$. We have
\[ (x_1 x_2)^y = N_{A/k}(x_1x_2)/N_{B/k}(x_1x_2) =
N_{A/k}(x_1)/N_{B/k}(x_1) \cdot N_{A/k}(x_2)/N_{B/k}(x_2) = x_1^y x_2^y, \] by
multiplicativity of the norm. In order to prove that
$x^{y_1 + y_2} = x^{y_1} x^{y_2}$ it is enough to show that if $A/k, B/k$ are
finite étale then $N_{A/k}(x|_A)N_{B/k}(x|_B) = N_{A \times B/k}(x|_{A \times
B})$, since $tr(A \times B) = tr(A) + tr(B)$. This follows from the first
paragraph of this proof. Since $tr(k) = 1$ we find $x^1 = N_{k/k}(x) = x$. Also
$x^0 = 1$ for all $x$, since $0 = tr(\emptyset)$ and the norm is multiplicative. It
remains to show that $x^{yz} = (x^y)^z$. Write $y = tr(A) - tr(B)$ and $z =
tr(A') - tr(B')$. Then
\begin{gather*} x^{yz} = x^{(tr(A) - tr(B))(tr(A')-tr(B'))} = x^{tr(A)tr(A') + tr(B)tr(B') -
tr(A)tr(B') - tr(A')tr(B)} \\ = \frac{x^{tr(A)tr(A')}
                           x^{tr(B)tr(B')}}{x^{tr(A)tr(B')} x^{tr(A')tr(B)}}
\end{gather*}
whereas
\begin{gather*}
   (x^y)^z = (x^{tr(A) - tr(B)})^{tr(A') - tr(B')} = (x^{tr(A)}/x^{tr(B)})^{tr(A') - tr(B')}
           \\= \frac{(x^{tr(A)})^{tr(A')} (x^{tr(B)})^{tr(B')}}{(x^{tr(B)})^{tr(A')} (x^{tr(A)})^{tr(B')}}
,\end{gather*}
by what we have already established.
It is thus enough to show that if
$A/k, B/k$ are finite étale then $x^{tr(A) tr(B)} = (x^{tr(A)})^{tr(B)}$. Note
that $tr(A \otimes_k B) = tr(A) tr(B)$; this follows for example from the
base change formula (for additive transfers). Let $t \in GW(A)$. Then by the base
change formula (see Definition \ref{def:tambara} part (4) and Corollary
\ref{corr:GW-tambara}) we get $(N_{A/k}(t))|_B = N_{A \otimes_k B/B}(t|_{A
\otimes_k B}).$ Substituting $t = x|_A$, applying $N_{B/k}$ and using
transitivity of the norm we get $N_{B/k}((N_{A/k}(x|_A))|_B) = N_{A \otimes_k
B/k}(x|_{A \otimes_k B})$. The left hand side is by definition
$(x^{tr(A)})^{tr(B)}$, whereas the right hand side is by definition $x^{tr(A
\otimes B)}$, which is $x^{tr(A)tr(B)}$. Thus we have established the desired
result.

(iii) It suffices to show that for $A/k$ finite étale we have $N_{A/k}(F_n
GW^\times(k)|_A) \subset F_n GW^\times(k)$. This is immediate from Proposition
\ref{prop:norm-basics} and the fact that restriction preserves the filtration
$F_n$.

(iv) We first establish the second claim. By (i) and (ii), both sides are linear
in $y_1$. We may thus assume that $y_1 = tr_{B/A}(1)$, for some $B/A$ finite étale.
Then $N_{A/k}((x_2|_A)^{y_1}) = N_{A/k}(N_{B/A}(x_2|_B))$ which equals
$N_{B/k}(x_2|_B)$ by transitivity, which is the same as $x_2^{tr_k(B)}$ by
definition. The second claim follows since $tr_k(B) = tr_{A/k}(tr_{B/A}(1))$, by
transitivity of transfer. For the first claim, we may assume that $y_2 = tr(C)$,
with $C/k$ finite étale. Then $y_2|_A = tr_{C \otimes_k A/A}(1)$, by the
base change formula (for additive transfers).
Thus $N_{A/k}(x_1^{y_2|_A}) = N_{A/k}(N_{C \otimes_k
A/A}(x_1|_{C \otimes_k A})).$ By transitivity of the norm, this is the same as
$N_{C \otimes_k A/k}(x_1|_{C \otimes_k A}) = N_{C/k}N_{C \otimes_k A/C}(x_1|_{C
\otimes_k A})$. By using the base change formula again, we deduce that $N_{C
\otimes_k A/C}(x_1|_{C \otimes_k A}) = N_{A/k}(x_1)|_C$. Putting everything
together, we find that \[ N_{A/k}(x_1^{y_2|_A}) = N_{C/k}N_{A/k}(x_1)|_C =
(N_{A/k}(x_1))^{y_2}.\] This was to be shown.
\end{proof}

\begin{prop} \label{prop:norm-not-so-basics}
\begin{enumerate}[(i)]
\item Suppose that $1 + x \in F_n GW^\times(k)$, and $y \in GW(k)$. Then $(1 + x)^y
  \equiv 1 + xy \pmod{I^{2n}(k)}$.
\item If $A/k$ is of degree 2, then $(-1)^{tr(A)} = tr(A) - 1$.
\item If $y \in I(k)$ then $(-1)^y \equiv 1 + y \pmod{I^2}.$
\item For any $n, m \ge 0$ we have
   $(F_n GW^\times(k))^{I^m(k)} \subset F_{n+m} GW^\times(k)$.
\end{enumerate}
\end{prop}
\begin{proof}
(i) Recall the map $\alpha_n: F_n GW^\times \to I^n/I^{2n}, x \mapsto x-1$ from equation
\eqref{eq:alphan}. What we are trying to show is equivalent to
$\alpha_n((1+x)^y) = \alpha_n(1+x)y$, i.e. that $\alpha_n$ is a $GW(k)$-module
homomorphism. This may be checked on an additive set of generators of $GW(k)$,
since we already know that $\alpha_n$ is a homomorphism of abelian groups.
Thus we may assume that $y = tr(A)$ for $A/k$ finite étale. In this case the
claim is immediate from
Proposition \ref{prop:norm-basics} part (iii) and the projection formula for
additive transfers.

(ii) Using Lemma \ref{lemm:wittkop-formula}, we compute \[ 0 = N_{A/k}(0) = N_{A/k}(1 +
(-1)) = 1 + N_{A/k}(-1) - tr_{A/k}(1).\] The result follows by rearranging.

(iii) Let $y \in GW(k)$. It follows from Remark \ref{rmk:degree-zero-Ntr} that we have
$dim((-1)^y) = (-1)^{dim(y)}$. Hence if $y \in I(k)$ then $dim((-1)^y) = 1$
and so $(-1)^y \in F_1 GW^\times(k)$. We now have the two maps $f, g: I(k) \to I(k)/I(k)^2$ given by
$f(y) = [y]$ and $g(y) = \alpha_1((-1)^y)$, and we wish to show that they are
equal. Both are group homomorphisms (the second on by Proposition
\ref{prop:GW-module-existence-basics}(ii)), so we need only check this on generators.
Generators of $I(k)$ are given by $tr(A) - 2$ for $A/k$ degree 2. Indeed we
know that any element of $GW(k)$ can be written as
\[ x = \sum_{i=0}^{n_1} tr(A_i) - \sum_{j=0}^{n_2} tr(B_j) + c \]
for degree 2 extensions $A_i, B_j/k$ and $c \in \{0, 1\}$. Then $dim(x) = 2(n_1-n_2) +
c = 0$ if and only if $c=0$ and $n_1 = n_2 =: n$. In this case we have
\[ x = \sum_{i=0}^n\left[(tr(A_i) - 2) - (tr(B_i) - 2)\right]. \]
Now if $A/k$ is degree $2$, then by (ii) we have $(-1)^{tr(A)-2} = (-1)^{tr(A)} = tr(A) - 1 = 1 + (tr(A)-2)$.
The claim follows.

(iv) The case $m = 0$ is Proposition \ref{prop:GW-module-existence-basics} part
(iii). The case $m > 1$ follows from $m=1$ and induction. So suppose $m=1$; i.e.
we need to show that $(F_nGW(k)^\times)^{I(k)} \subset F_{n+1} GW^\times(k)$.
For $n \ge 1$ this is immediate from (i). We now deal with $n=0$.
Thus let $x \in GW^\times(k)$. Then $dim(x) \in \{\pm 1\}$. Note that $x \in
F_1GW^\times(k)$ if and only if $dim(x) = 1$. Let $y \in I(k)$. Then $dim(y)=0$,
whence $dim(x^y) = dim(x)^y = 1$ as a consequence of Remark \ref{rmk:degree-zero-Ntr}. It follows
that $x^y \in F_1GW^\times(k)$.  This concludes the proof.
\end{proof}

\begin{rmk} \label{rmk:sqrt-2}
It follows from part (iv) of Proposition \ref{prop:norm-not-so-basics} that
$(-1)^{I^2(k)} \subset F_2 GW^\times(k)$.
If $\sqrt{2} \in k$ then one may show that actually $(-1)^{I^2(k)} = 1$. This
follows from Lemma \ref{lemm:identify-1-exp} and Theorem
\ref{thm:log-module-transfers}. In fact the proof of Lemma
\ref{lemm:identify-1-exp} shows that $(-1)^{(\lra{a} - 1)(\lra{b} - 1)} = 1 +
(\lra{2} - 1)(\lra{a} - 1)(\lra{b} - 1)$, from which we also see that
$(-1)^{I^2(k)} \ne 1$ in general.
\end{rmk}

We can use the results of this section to give a presentation of
$GW^\times(k)$.

\begin{prop} \label{prop:GW-times-presentation}
The $GW(k)$-module $GW^\times(k)$ is generated by $F_2 GW^\times(k)$ and $-1$.
Moreover the following sequence is exact:
\[ I^2(k) \xrightarrow{p \oplus q} GW(k)/2 \oplus F_2 GW^\times(k) \xrightarrow{r/s}
      GW^\times(k) \to 1 \]
Here $p: I^2(k) \hookrightarrow GW(k) \to  GW(k)/2$ is the canonical map,
$q(y) = (-1)^y$, $r(x) = (-1)^x$, $s: F_2 GW^\times(k) \hookrightarrow
GW^\times(k)$ is the canonical inclusion, and $(r/s)(x) := r(x)/s(x)$
\end{prop}
\begin{proof}
Note that $(-1)^2 = 1$ so $r$ is well-defined. Moreover if $y \in I^2(k)$ then
$(-1)^y \in F_2 GW^\times(k)$ by Proposition \ref{prop:norm-not-so-basics} part
(iv) (or (iii)), so $q$ is well-defined.

To show the claim about generation, or equivalently surjectivity of $r/s$, it
suffices to show that any $x \in GW^\times(k)$ can be written as $(-1)^y z$,
with $y \in GW(k)$ and $z \in F_2 GW^\times(k)$. Certainly $dim(x)= \pm 1$, so
$x = (-1)^n z_1$ for some $n \in \ZZ$ and $z_1 \in F_1 GW^\times(k)$. Now
$\alpha_1(z_1) \in I(k)/I^2(k)$. Pick $t \in I(k)$ with $[t] = -\alpha_1(z_1)$.
Then $\alpha_1[(-1)^{t} z_1] = 0$ by Proposition
\ref{prop:norm-not-so-basics} part (iii), and so $(-1)^{t} z_1\in F_2 GW^\times(k)$.
Consequently $x = (-1)^{n - t}((-1)^t z_1)$
is of the required form.

It remains to verify exactness in the middle. It is clear that the composite of
the two maps is 0. Now let $x \in GW(k)/2$. It suffices
to show that $(-1)^x \in F_2 GW^\times(k)$ only if $x$ is in the image of
$p$. Hence suppose that $(-1)^x \in F_2 GW^\times(k)$. Then $1 = dim((-1)^x) =
(-1)^{dim(x)}$ and so $dim(x)$ is even, whence we may assume that $x \in I$. Now
$0 = \alpha_1((-1)^x) \equiv x \pmod{I^2}$ by Proposition
\ref{prop:norm-not-so-basics} part (iii) again, and so $x \in I^2$. This
concludes the proof.
\end{proof}

\begin{rmk}
For a more optimal form of this proposition, see Theorem \ref{thm:final}.
\end{rmk}

\section{The sheaf $\ul{GW}^\times$ and the homotopy module $F_*$}
\label{sec:sheaf-GW-times}
Throughout this section, unless stated otherwise, the field $k$ is assumed perfect. As always, we assume
that $char(k) \ne 2$. We will use various results about strictly homotopy
invariant sheaves and homotopy modules. Confer Appendix  \ref{sec:app-htpy-mod}
for some recollections on this material.

Recall that the construction $X \mapsto GW(X)$ defines a presheaf on the
category of schemes. In the previous section, we have studied its restriction to
the subcategory $\Fet/k$ of finite étale $k$-schemes. From now on, we will study
it on all of $Sm(k)$, the category of \emph{smooth} (separated) $k$-schemes.
The associated sheaf (in the Nisnevich or Zariski topology) is denoted
$\ul{GW}$,
is called the sheaf of unramified Grothendieck--Witt groups, and is strictly
homotopy invariant. Recall that a sheaf $F$ on $Sm(k)$ is called strictly
homotopy invariant if the canonical map $H^p_{Nis}(X, F) \to H^p_{Nis}(X \times \Aone, F)$
is an isomorphism for all $X \in Sm(k)$ and all $p \ge 0$.
The sheaf $\ul{GW}$ coincides with the sheaf constructed by Morel \cite[Section
3.2]{A1-alg-top} \cite[Theorem A]{panin-witt-purity}.  Note also
that $\ul{GW}|_{\Fet/k} = GW|_{\Fet/k}$, since all finite étale $k$-schemes are
(finite) disjoint unions of Nisnevich local schemes. Since $\ul{GW}$
is a sheaf of rings, it has a subsheaf of units which we denote
$\ul{GW}^\times$.
This
is the sheaf associated with $Sm(k) \ni X \mapsto GW(X)^\times$. Our first task
is to prove that $\ul{GW}^\times$ is also strictly homotopy invariant.

In order to do this, we recall that there are the sheaves of ideals $\ul{I}^n
\subset \ul{GW}$. We define a filtration of $\ul{GW}^\times$ via $F_n
\ul{GW}^\times (X) = (1 + \ul{I}^n(X)) \cap \ul{GW}^\times(X)$. As before we get
homomorphisms $\alpha_n: F_n \ul{GW}^\times \to \ul{I}^n/\ul{I}^{2n}$, where on
the right hand side we mean the quotient taken in the category of Nisnevich
sheaves.

If $F$ is any
(pre)sheaf on $Sm(k)$ we write $F_{tor}$ for the (pre)sheaf $F_{tor}(X) =
F(X)_{tor}$, where for an abelian group $A$ we write $A_{tor}$ for the
subgroup of torsion elements. It is strictly homotopy invariant if $F$ is. This follows from the
fact that the category of strictly homotopy invariant sheaves is abelian and
closed under filtered colimits.

\begin{lemm} \label{lemm:Fn-tors}
If $n \ge 2$ then $F_n \ul{GW}^\times = 1 + \ul{I}^n_{tor}$, where the
identification holds as sub-presheaves of $\ul{GW}$.
\end{lemm}
\begin{proof}
Let $1 + x \in 1 + \ul{I}^n(X) \subset \ul{GW}(X)$. We need to show that
$1+x \in \ul{GW}^\times(X)$ if and only if $x$ is torsion. I claim that $x$ is
torsion if and only if it is nilpotent. Indeed since $\ul{GW}$ is strictly
homotopy invariant it is unramified \cite[Lemma 6.4.4]{morel2005stable}, and thus it suffices to prove
the claim for $GW(L)$ with $L$ a field, where it follows from
\cite[Theorems III.3.6 and III.3.8]{milnor1973symmetric}. We thus need to show
that $1+x \in \ul{GW}(X)$ is invertible if and only if $x$ is nilpotent.
Certainly if $x$ is nilpotent then $1+x$ is invertible. Conversely, if $1+x$ is
invertible then so is its image in $GW(L)$ for any field $L$, and then by
unramifiedness of $\ul{GW}$ again it suffices to prove: if $1+x \in GW(L)$ is
invertible with $x \in I^n(L)$ and $n \ge 2$, then $x$ is nilpotent (or
equivalently, torsion).  
Let $\sigma: GW(L) \to \ZZ$ be a signature map. By \cite[Theorems III.3.6 and
III.3.8]{milnor1973symmetric} again it suffices to show that $\sigma(x) = 0$.
But $\sigma(I) \subset 2\ZZ$ and hence $\sigma(x) \in 2^n \ZZ$, whereas also
$\sigma(1 + x) = 1 + \sigma(x) \in \ZZ^\times = \{\pm 1\}$. As $n \ge 2$ this
implies that $\sigma(x) = 0$, as was to be shown.
\end{proof}

\begin{lemm} \label{lemm:GW-times-subquotients}
We have
\[ \ul{GW}^\times / F_1 \ul{GW}^\times \iso \ZZ/2 \iso \ul{k}_0^M, \]
induced by $\dim: \ul{GW}^\times \to \ZZ$ (here by $\ZZ$ and $\ZZ/2$ we also
denote the associated constant sheaves),
\[ F_1 \ul{GW}^\times / F_2 \ul{GW}^\times \iso \Gm/2 \iso \ul{k}_1^M \iso
\ul{I}/\ul{I}^2, \]
induced by $\alpha_1$, and for $n
\ge 2$ we have
\[ F_n \ul{GW}^\times / F_{n+1} \ul{GW}^\times \iso
\ul{I}^n_{tor}/\ul{I}^{n+1}_{tor} \hookrightarrow \ul{I}^n/\ul{I}^{n+1}, \]
induced by $\alpha_n$. In
particular all of the subquotients of the filtration are strictly homotopy
invariant.
\end{lemm}
\begin{proof}
For $n = 0$ the map $\ul{GW}^\times/F_1 \ul{GW}^\times \to
(\ul{GW}/\ul{I})^\times = \ZZ/2$ is an isomorphism: it is surjective since it
has a section and it is injective because $F_1 \ul{GW}^\times = (1 + \ul{I})
\cap \ul{GW}^\times$ by definition.

For $n \ge 1$ the map $\alpha_n$ satisfies
$\alpha_n^{-1}(\ul{I}^{n+1}/\ul{I}^{2n}) = F_{n+1} \ul{GW}^\times$ and hence
induces an injection $\beta_n: F_n \ul{GW}^\times/F_{n+1} \ul{GW}^\times \hookrightarrow
\ul{I}^n/\ul{I}^{n+1}$.

There is a homomorphism $\Gm/2 \to \ul{GW}^\times, a \mapsto \langle a \rangle$
splitting $\beta_1$, so $\beta_1$ is an isomorphism. For $n \ge 2$ by Lemma
\ref{lemm:Fn-tors} we have $F_n \ul{GW}^\times = 1 + \ul{I}^n_{tor}$ and so the
image of $\beta_n$ is $\ul{I}^n_{tor}/\ul{I}^{n+1}_{tor} \subset
\ul{I}^n/\ul{I}^{n+1}$, as claimed.

For the last claim, since each $\ul{I}^n$ is strictly homotopy invariant
\cite[Example 3.34]{A1-alg-top} so is $\ul{I}^n_{tor}$, and hence so is the
quotient $\ul{I}^n_{tor}/\ul{I}^{n+1}_{tor}$. Here we have used again that the
category of strictly homotopy invariant sheaves is abelian and closed under
filtered colimits.
\end{proof}

We will repeatedly use the following result, essentially due to Elman and Lum.

\begin{lemm}[Elman and Lum \cite{elman-lum-2cohom}]
   \label{lemm:Itors-vanishing}
Let $k$ be a field such that $char(k) \ne 2$
and $vcd_2(k) < n$.
Then $I^n_{tor}(k) = 0$ and in particular $2^n I^r_{tor}(k) = 0$ for all $r
> 0$ (and also $2^n W_{tor}(k) = 0$).
\end{lemm}
\begin{proof}
Applying (vi) of the last theorem of \cite{elman-lum-2cohom} to $K = k(T), F =
k$ gives the first statement. The
remainder follows from $2 \in I(k) \subset W(k)$.
\end{proof}

\begin{thm} \label{thm:GW-times-strictly-htpy-inv}
Let $k$ be any field with $char(k) \ne 2$.
Then the sheaf $\ul{GW}^\times$ (on $Sm(k)$) is strictly homotopy
invariant. The same is true for $F_r \ul{GW}^\times$ for any $r$.
\end{thm}
\begin{proof}
Suppose first that $vcd_2(k) < n$ and $k$ is perfect.
Then for any field $L/k$ of
transcendence degree at most $m$ over $k$
we have $vcd_2(k) < n+m$ \cite[Theorem 28 of Chapter 4]{shatz1972profinite}
and so $I^{n+m}_{tor}(L) = 0$ by Lemma \ref{lemm:Itors-vanishing}.

Let $X \in Sm(k)$ be of dimension at most $m$. It follows from the first
paragraph and unramifiedness
that $F_{n+m} \ul{GW}^\times |_{X_{Nis}} = 1$. Hence on $X$ (and on $X \times
\Aone$) the sheaf $\ul{GW}^\times$ is a \emph{finite} extension of strictly
homotopy invariant sheaves, by Lemma \ref{lemm:GW-times-subquotients}, and
consequently is strictly homotopy invariant.

The same argument works for $F_r \ul{GW}^\times$ for $r \ne 0$.

For the general case
in which $vcd_2(k)$ might be infinite and $k$ need not be perfect we use a continuity
argument. Let $k_0 \subset k$ be the prime subfield and write $p: Spec(k) \to
Spec(k_0)$ for the canonical morphism. Then $k_0$ is perfect and $vcd_2(k) <
\infty$. The morphism $p$ is essentially smooth by \cite[Lemma
A.2]{hoyois-algebraic-cobordism}. Hence by Lemma \ref{corr:app-cont} we find
that $\ul{GW}|_{Sm(k)} = p^* (\ul{GW}|_{Sm(k_0)})$ and thus also $F_r
\ul{GW}^\times |_{Sm(k)} = p^* (F_r \ul{GW}^\times |_{Sm(k_0)})$. Since $p^*$
preserves strictly homotopy invariant sheaves \cite[Lemma
A.4]{hoyois-algebraic-cobordism}, this concludes the proof.
\end{proof}



\begin{prop} \label{prop:GW-mod-str-existence-extension}
  There exists a unique structure of a
  $\ul{GW}$-module on $\ul{GW}^\times$
  such that for a field $L/k$ of finite transcendence degree, the induced
  $GW(L)$-module structure on $GW^\times(L)$ is the one from Section
  \ref{sec:GW-module-structure}.
\end{prop}
\begin{proof}

Uniqueness follows from unramifiedness of $\ul{GW}$.
For existence, let $x \in \ul{GW}(X)$ and $y
\in \ul{GW}^\times(X)$. Write $a: X^{(0)} \to X$ for the inclusion of the
generic points. We need to show that $(a^*y)^{a^*x} \in \ul{GW}^\times(X)
\subset \ul{GW}^\times(X^{(0)})$ is an unramified element. Indeed, recall
(possibly from Appendix \ref{sec:app-htpy-mod}) that if $F$ is an unramified
sheaf and $X$ is connected (hence irreducible), we have
\[ F(X) = \bigcap_{x \in X^{(1)}} F(X_x) \subset F(k(X)). \]
In other words we need to prove that $(a^*y)^{a^*x} \in \ul{GW}^\times(X)$
whenever $X$ is the spectrum of a dvr (or more generally local ring).
In this case $\ul{GW}(X) = GW(X)$ is generated by
the one-dimensional diagonal forms $\langle a \rangle$ with $a \in
\mathcal{O}^\times(X)$ \cite[Corollary I.3.4]{milnor1973symmetric} and
consequently the traces of étale $X$-schemes generate $GW(X)$, by the same
argument as before Proposition \ref{prop:GW-module-existence-basics}. Let $x = tr(Y_1) -
tr(Y_2)$. Then $(a^*y)^{a^*x} = a^*(N_{Y_1/X}(y|_{Y_1})/N_{Y_2/X}(y|_{Y_2}))$. Since
$N_{Y_1/X}(y|_{Y_1}), N_{Y_2/X}(y|_{Y_2}) \in GW(X)^\times$, this concludes the proof.
\end{proof}

By unramifiedness, the results from Section \ref{sec:GW-module-structure} over
fields immediately extend to all of $\ul{GW}$:

\begin{corr} \label{corr:not-so-basics}
\begin{enumerate}[(i)]
\item Each of the subsheaves $F_n \ul{GW}^\times \subset \ul{GW}^\times$ is a
  sub-$\ul{GW}$-module.
\item For any $n, m \ge 0$ we have
   $(F_n \ul{GW}^\times)^{\ul{I}^m} \subset F_{n+m} \ul{GW}^\times$.
\end{enumerate}
\end{corr}
\begin{proof}
For (i), if $x \in \ul{GW}(X)$ and $y \in F_n\ul{GW}^\times(X)$ then we
wish to show that $y^x \in F_n \ul{GW}^\times(X)$, where $y^x \in GW^\times(X)$ is defined using
the module structure established in Proposition
\ref{prop:GW-mod-str-existence-extension}. But by definition
$F_n \ul{GW}^\times(X) = 1 + \ul{I}^n(X) \cap \ul{GW}^\times(X)$,
and $\ul{I}^n(X) = \ul{I}^n(X^{(0)}) \cap \ul{GW}(X)$ as a consequence of Lemma
\ref{lemm:app-submodule}.
It follows that $F_n \ul{GW}^\times(X) =
\ul{GW}^\times(X) \cap F_n GW^\times(X^{(0)})$. We are thus
reduced to showing that $y^x \in F_n GW^\times(X^{(0)})$. This is Proposition
\ref{prop:GW-module-existence-basics} part (iii).

The argument for (ii) is the same, using Proposition
\ref{prop:norm-not-so-basics} part (iv).
\end{proof}

We can use the $\ul{GW}$-module structure to define a pairing
\[ \beta: \ZZ[\Gm] \otimes \ul{GW}^\times \to \ul{GW}^\times,
    (u, x) \mapsto x^{\lra{u} - 1}. \]
Since $\lra{u} - 1 \in \ul{I}$, by Corollary \ref{corr:not-so-basics}
we know that $\beta(\ZZ[\Gm] \otimes F_n \ul{GW}^\times) \subset F_{n+1}
\ul{GW}^\times$. We write $\beta_n: \ZZ[\Gm] \otimes F_n \ul{GW}^\times \to F_{n+1}
\ul{GW}^\times$ for this restricted pairing.
By adjunction, $\beta_n$ induces a homomorphism $\beta_n^\dagger: F_n
\ul{GW}^\times \to (F_{n+1} \ul{GW}^\times)_{-1}$. Here for a presheaf $F$ we
denote by $F_{-1} = \iHom(\ZZ[\Gm], F)$ its \emph{contraction}; see Appendix
\ref{sec:app-htpy-mod} for more on this construction.

\begin{prop} \label{prop:betan-iso}
Let $k$ be a field of characteristic different from $2$.
For $n \ge 2$ the homomorphism $\beta_n^\dagger: F_n
\ul{GW}^\times \to (F_{n+1} \ul{GW}^\times)_{-1}$ is an isomorphism.
\end{prop}
\begin{proof}
Throughout we will assume $n \ge 2$. We first assume that $k$ is perfect and
$vcd_2(k) < \infty$; these assumptions are removed at the end by a continuity
argument.
The commutative square
\begin{equation*}
\begin{CD}
\ZZ[\Gm] \otimes F_{n+1} \ul{GW}^\times  @>\beta_{n+1}>> F_{n+2} \ul{GW}^\times \\
@VVV                                  @VVV                 \\
\ZZ[\Gm] \otimes F_{n} \ul{GW}^\times  @>\beta_n>> F_{n+1} \ul{GW}^\times \\
\end{CD}
\end{equation*}
in which the vertical maps are the canonical inclusions,
induces by adjunction a commutative square
\begin{equation*}
\begin{CD}
F_{n+1} \ul{GW}^\times  @>\beta_{n+1}^\dagger>> (F_{n+2} \ul{GW}^\times)_{-1} \\
@VVV                                  @VVV                 \\
F_{n} \ul{GW}^\times  @>\beta_n^\dagger>> (F_{n+1} \ul{GW}^\times)_{-1} \\
\end{CD}
\end{equation*}
Since contraction is an exact operation \cite[Lemma 7.33]{A1-alg-top}, by Lemma
\ref{lemm:GW-times-subquotients} we get a diagram of short exact sequences
\begin{equation}
\label{eq:first-exact-sequences}
\begin{CD}
1 @>>> F_{n+1} \ul{GW}^\times @>>> F_{n} \ul{GW}^\times @>>> \ul{I}^n_{tor}/\ul{I}^{n+1}_{tor} @>>> 0 \\
@.       @V{\beta_{n+1}^\dagger}VV      @V{\beta_n^\dagger}VV @V{\gamma_n^\dagger}VV \\
1 @>>> (F_{n+2} \ul{GW}^\times)_{-1} @>>> (F_{n+1} \ul{GW}^\times)_{-1} @>>> (\ul{I}^{n+2}_{tor}/\ul{I}^{n+1}_{tor})_{-1} @>>> 0. \\
\end{CD}
\end{equation}
Here $\gamma_n^\dagger$ is defined so as to make the diagram commute.

I claim that $\gamma_n^\dagger$ is an isomorphism. To see this, let $\delta_n:
\ZZ[\Gm] \otimes \ul{I}^n \to \ul{I}^{n+1}$ be the homomorphism $(u, x) \mapsto
(\lra{u} - 1)x$. Then $\delta_n^\dagger: \ul{I}^n \to (\ul{I}^{n+1})_{-1}$ is an
isomorphism. This is just because $\ul{I}^*$ is a homotopy module where the
element $[u] \in \ul{K}_1^{MW}$ acts via $\lra{u} - 1 \in \ul{I}^1$.
It follows that the restriction $\delta_n^\dagger:
\ul{I}^n_{tor} \to (\ul{I}^{n+1}_{tor})_{-1}$ is also an isomorphism. Note
that $(F_{tor})_{-1} = (F_{-1})_{tor}$. Now
consider the commutative diagram of exact sequences
\begin{equation*}
\begin{CD}
1 @>>> \ul{I}_{tor}^{n+1} @>>> \ul{I}_{tor}^n @>>> \ul{I}^n_{tor}/\ul{I}^{n+1}_{tor} @>>> 0 \\
@.       @V{\delta_{n+1}^\dagger}VV      @V{\delta_n^\dagger}VV @V{\epsilon_n^\dagger}VV \\
1 @>>> (\ul{I}_{tor}^{n+2})_{-1} @>>> (\ul{I}_{tor}^{n+1})_{-1} @>>> (\ul{I}^{n+1}_{tor}/\ul{I}^{n+2}_{tor})_{-1} @>>> 0. \\
\end{CD}
\end{equation*}
We find that $\epsilon_n^\dagger$ is an isomorphism, by the 5-lemma. But also $\epsilon_n^\dagger =
\gamma_n^\dagger$. For this it suffices to show that $\epsilon_n = \gamma_n$.
Since the target is strictly homotopy invariant, hence
unramified, it suffices to show that $\epsilon_n$ and $\gamma_n$ induce the same
map on sections over fields. Thus let $L$ be a field. By definition, the
following diagram commutes:
\begin{equation*}
\begin{CD}
\ZZ[L^\times] \otimes F_n GW^\times(L) @>>> \ZZ[L^\times] \otimes I^n_{tor}(L)/I^{n+1}_{tor}(L) \\
@V{\beta_n}VV                                @V{\gamma_n}VV \\
F_{n+1}GW^\times(L)                    @>>> I^{n+1}_{tor}(L)/I^{n+2}_{tor}(L).
\end{CD}
\end{equation*}
Here the horizontal maps are induced by $x \mapsto x-1$, and $\beta_n(a \otimes
x) = x^{\lra{a} - 1}$, for $a \in L^\times$ and $x \in F_n GW^\times(L)$. It
now follows from Proposition
\ref{prop:norm-not-so-basics} part (i) that for $\bar{x} \in
I^n_{tor}(L)/I^{n+1}_{tor}(L)$ we have $\gamma_n(a \otimes \bar{x}) = (\lra{a} -
1) x$. Here we use that $n \ge 2$. By definition, this is the same as
$\epsilon_n(a \otimes \bar{x})$.
Hence $\epsilon_n = \gamma_n$ and thus $\gamma_n^\dagger$ is an isomorphism as claimed.

Now in order to show that $\beta_n^\dagger$ is an isomorphism, it suffices to
show that for every field $L$ (of finite transcendence degree over $k$) the
section $\beta_n^\dagger(L)$ is an isomorphism (since the kernel and cokernel of
$\beta_n^\dagger$ are strictly homotopy invariant and hence unramified). Recall
that we assume for now that $vcd_2(k) < \infty$. Then also $vcd_2(L) < \infty$ and
for $n$ sufficiently large we have $F_n
GW^\times(L) = 1$, by the Lemmas \ref{lemm:Fn-tors} and
\ref{lemm:Itors-vanishing}.
In particular for $n$ sufficiently large
$\beta_n^\dagger(L)$ is an isomorphism. We may thus prove that
$\beta_n^\dagger(L)$ is an isomorphism for all $n \ge 2$ by descending induction
on $n$. The induction step follows by considering the diagram of exact sequences
\eqref{eq:first-exact-sequences} and using that $\gamma_n^\dagger$ is
an isomorphism, as we established above.

For the general case in which $vcd_2(k)$ may be infinite and $k$ may be
imperfect, let $p: Spec(k) \to
Spec(k_0)$ be an essentially smooth morphism to a perfect field with
$vcd_2(k_0) < \infty$ (e.g. $k_0$ the prime field). It follows from Lemma
\ref{lemm:Nis-cont} that $p^*$ commutes with contractions, and it follows from
Corollary \ref{corr:app-cont} that $p^* F_n \ul{GW}^\times = F_n
\ul{GW}^\times$. Then $\beta_n^\dagger = p^* \beta_n^\dagger$ is an isomorphism.
\end{proof}

We have thus managed to deloop the sheaves $F_n \ul{GW}^\times$ for $n \ge 2$.
Recall (possibly from Appendix \ref{sec:app-htpy-mod})
that a \emph{homotopy module} consists of a sequence of sheaves $F_n \in
Shv_{Nis}(Sm(k))$ together with isomorphisms $F_n \iso (F_{n+1})_{-1}$, such
that each $F_n$ is strictly homotopy invariant.

\begin{corr} \label{corr:F*-existence}
There is a homotopy module $F_*$, determined up to unique isomorphism of
homotopy modules,  such that for $n
\ge 2$ the following hold:
\begin{itemize}
\item $F_n \iso F_n \ul{GW}^\times$, and
\item the bonding map $F_n \xrightarrow{\iso} (F_{n+1})_{-1}$ equals $\beta_n^\dagger$.
\end{itemize}
\end{corr}

\section{The logarithm isomorphism}
\label{sec:logarithm}
Throughout this section, the base field $k$ is assumed to be perfect. As always,
we assume that it is of characteristic different from 2.

In this section we shall study in more detail the homotopy module $F_*$. Recall
from Appendix \ref{sec:app-htpy-mod}
that if $G_*$ is any homotopy module, then each $G_n$ has the structure of a
$\ul{GW}$-module,
and also has transfers along finite étale
morphisms known as \emph{cohomological transfers}. Note that the definition of a
homotopy module $G_*$ only asks for isomorphisms $G_n \xrightarrow{\iso}
(G_{n+1})_{-1}$ and strict homotopy invariance of the $G_i$. The transfers and
$\ul{GW}$-module structure are implicit in this data.
In particular, each of the sheaves $F_n \ul{GW}^\times$ (for $n \ge 2$) acquires
an a priori \emph{new} $\ul{GW}$-module structure and \emph{new} transfers. In
this section, among other things, we shall show
that the ``new'' $\ul{GW}$-module structure on $F_n = F_n \ul{GW}^\times$ coincides with that
of Section \ref{sec:sheaf-GW-times}, and that the ``new'' cohomological transfers
coincide with Rost's multiplicative transfers. We begin by comparing the module
structures. For $X \in Sm(k), a \in \ul{GW}(X), x \in F_*(X)$ we denote by $ax$ the
action of $\ul{GW}(X)$ coming from the fact that $F_*$ is a homotopy module, and
we denote by $x^a$ the action coming from the module structure we constructed in
Section \ref{sec:sheaf-GW-times}. What we are trying to prove, then, is that
$ax = x^a$.

\begin{lemm}
For $n \ge 2$ the $\ul{GW}$-module structure on $F_n = F_n \ul{GW}^\times$
coincides with the module structure defined in Section \ref{sec:sheaf-GW-times}.
\end{lemm}
\begin{proof}
We first describe the $\ul{GW}$-module structure on a homotopy module $G_*$.
Let $L$ be a field of finite transcendence degree over $k$.
Let $\mathcal{O}
\subset L$ be a geometric dvr with uniformiser $\pi$ and residue field $\kappa$.
As explained in Appendix \ref{sec:app-htpy-mod},
there is the boundary map $\partial^\pi: G_{n+1}(L) \to G_n(\kappa)$, with
kernel $G_{n+1}(\mathcal{O})$.
The isomorphism $G_n \to
(G_{n+1})_{-1}$ induces a map $\ZZ[\Gm] \otimes G_n \to G_{n+1}$ which we denote
$(u, x) \mapsto [u]x$. Then by Lemma \ref{lemm:app-bdry-formula} in the appendix one has
\begin{equation} \label{eq:htpy-module-module-str-partial}
  \partial^\pi([u\pi]x) = \lra{s^*(u)}s^*(x),
\end{equation}
where $s: Spec(\kappa) \to Spec(\mathcal{O})$ is the inclusion of the closed
point, and where on the right hand side multiplication by an element of
$GW(\kappa)$ refers to the $GW$-module structure we are describing.
If in addition $\mathcal{O}$ is Henselian then $s$ has a section and so
$s^*$ is surjective, so in this case equation
\eqref{eq:htpy-module-module-str-partial} determines the $GW$-module structure
on $G_*(\kappa)$ uniquely.

Now let $\kappa$ be a field of finite transcendence degree over $k$. It
suffices to show that the $GW(\kappa)$-module structure on $F_n$ is the
one from Section \ref{sec:sheaf-GW-times}. Choose an essentially smooth local curve over $k$ with residue
field $\kappa$ (e.g. the localisation of $\mathbb{A}^1_\kappa$ in the origin). Passing
to the Henselization, we obtain a Henselian dvr $\mathcal{O}$ with residue field
$\kappa$ and some fraction field $L$, also of finite transcendence degree. Pick a uniformiser $\pi$.
We wish to show that for all $\bar{u} \in
\kappa^\times$ and all $\bar{x} \in F_n(\kappa)$ we have $\lra{\bar{u}}\bar{x} =
\bar{x}^{\lra{\bar{u}}},$
where on the left hand side we mean the module structure coming from $F_*$
being a homotopy module and on the right hand side we mean the module
structure constructed in Section \ref{sec:sheaf-GW-times}.
By the first paragraph, for this it suffices to show that
$\partial^\pi([u\pi]x) =
s^*(x)^{\lra{\bar{u}}}$ for all $x \in F_n(\mathcal{O})$ and all $u \in
\mathcal{O}^\times$. We compute
\[ [u\pi]x = x^{\lra{u\pi}-1} = x^{\lra{u} \lra{\pi} - \lra{u} + \lra{u} - 1} =
   \left(x^{\lra{u}}\right)^{\lra{\pi} - 1} x^{\lra{u} - 1} = ([\pi] x^{\lra{u}})([u]x). \]
Here the first and last equalities are by definition of the homotopy module
structure on $F_*$.
Note that $\partial^\pi([\pi]z) = s^*(z)$ by equation
\eqref{eq:htpy-module-module-str-partial}, and $\partial^\pi$ is a homomorphism
with kernel $F_{n+1}(\mathcal{O})$. Since $[u]x \in F_{n+1}(\mathcal{O})$ we thus obtain
\[ \partial^\pi([u\pi]x) = \partial^\pi([\pi] x^{\lra{u}}) = s^*(x^{\lra{u}}) =
s^*(x)^{\lra{\bar{u}}}. \]
This concludes the proof.
\end{proof}

What we have done so far has the following interesting consequence.

\begin{lemm} \label{lemm:defn-log-approx}
Let $n \ge 2, m \ge 0$, $X \in Sm(k)$, and $x \in F_n \ul{GW}^\times(X)$. Then
the element
\[ x^{(\lra{t_1} - 1)(\lra{t_2} - 1) \dots (\lra{t_m} - 1)} \in
         F_{n+m}\ul{GW}^\times(X \times (\mathbb{A}^1 \setminus 0)^m)
         \subset \ul{GW}(X \times (\mathbb{A}^1 \setminus 0)^m) \]
may be written as
\[ 1 + (\lra{t_1} - 1)(\lra{t_2} - 1) \dots (\lra{t_m} - 1) y \]
for a unique $y \in \ul{I}^n(X)$.

This induces a bijection (of sets!) $\log_{(m)}: F_n \ul{GW}^\times(X) \to
  \ul{I}^n_{tor}(X)$.
\end{lemm}
\begin{proof}
Recall that if $G_*$ is any homotopy module, then
\begin{equation*}
  G_{n+1}(X \times (\Aone \setminus 0)) = G_{n+1}(X) \oplus [t_1]G_n(X),
\end{equation*}
where $G_n(X) \to [t_1]G_n(X)$ is an
isomorphism, namely multiplication by $t_1$.
Moreover in this decomposition, the factor
$[t_1]G_n(X)$ consists precisely of those $x \in G_{n+1}(X \times (\Aone
\setminus 0))$ such that $i_1^*(x) = 0$, where $i_1: X \to X \times (\Aone
\setminus 0)$ is the inclusion of the point $1 \in \Aone$. This is the
content of Lemma
\ref{lemm:app-contractions} in the appendix.
By induction, the
map $[t_1] \dots [t_m]: G_n(X) \to G_n(X \times (\Aone \setminus 0)^m)$ is
injective, and its image consists of precisely those $x \in G_n(X \times (\Aone
\setminus 0)^m)$ such that for each $r \in \{1, \dots, m\}$ we have $j_r^*(x) =
0$, where $j_r: X \times (\Aone \setminus 0)^{m-1} \to X \times (\Aone \setminus
0)^{m}$ is the inclusion at the point $1$ in the $r$-th factor $\Aone
\setminus 0$.

Applying this to the homotopy module $F_*$ we find that
\[ \alpha: F_n(X) \to F_{n+m}(X \times (\Aone \setminus 0)^m),
    x \mapsto [t_1]\dots[t_m] x = x^{(\lra{t_1} - 1)\dots(\lra{t_m} - 1)} \]
is an injection with image consisting of those
$y \in F_{n+m}(X \times (\Aone \setminus 0)^m)
  \subset \ul{GW}(X \times (\Aone \setminus 0)^m)$ such that
$j_r^*(y) = 1$ for all $r \in \{1, 2, \dots, m\}$. Since $j_r^*$ is a ring homomorphism, we conclude
by Lemma \ref{lemm:Fn-tors} that $\alpha - 1$ is a bijection onto the subset of
$\ul{I}^{n+m}_{tor}(X \times (\Aone \setminus 0)^m)$ consisting of those elements such
that $j_r^* = 0$ for all $r$. Applying the remark from the first paragraph to
the homotopy module $\ul{I}^*_{tor}$ (for which $[u]x = (\lra{u} - 1)x$) concludes the
proof.
\end{proof}

Thus for any $x \in F_n \ul{GW}^\times(X)$ we obtain a sequence $x = \log_{(0)}(x),
\log_{(1)}(x), \log_{(2)}(x), \dots \in \ul{I}^n_{tor}(X)$. We would like to take the
``limit'' of this sequence.

For the remainder of this section, we will use the abbreviation $P_m :=
\prod_{i=1}^m (\lra{t_i} - 1) \in \ul{GW}((\Aone \setminus 0)^m)$.
\begin{rmk} \label{rmk:classical-analysis-log}
By Lemma \ref{lemm:defn-log-approx} we have $x^{P_m} = 1 + P_m \log_{(m)}(x)$. Since
multiplication by $P_m$ is injective in an appropriate sense (see the previous proof), we
may write this as
\[ \log_{(m)}(x) = (x^{P_m} - 1)/P_m. \]
We think of $P_m \in I^m$ as small, and taking the limit we propose corresponds
to the formula
\[ \lim_{\epsilon \to 0+} (x^{\epsilon} - 1)/\epsilon = \log(x) \]
from classical analysis.
\end{rmk}

\begin{lemm} \label{lemm:technical-1}
Let $L'/L$ be a degree 2 extension ($L$ of characteristic not 2).
\begin{enumerate}[(i)]
\item We have $N_{L'(t)/L(t)}(\lra{t} - 1) = -tr(L')(\lra{t} - 1)$.
\item We have $tr(L')^2 = 2tr(L')$.
\item There exists $y \in GW(L)$ such that $yN_{L'/L}(2) = 8$.
\end{enumerate}
\end{lemm}
\begin{proof}
(i) By Lemma \ref{lemm:wittkop-formula} we get $N(\lra{t} - 1) =
N(\lra{t}) - tr(1)\lra{t} + N(-1)$, where we have used that
$tr(-\overline{\lra{t}}) = -\lra{t}tr(1)$. Note that $N(\lra{t}) = \lra{t^2} =
1$:
this is clear if $A = k \times k$, and for the quadratic case see \cite[Lemma
2.6(ii)]{wittkop2006multiplikative}. Since $N(-1) = (-1)^{tr(1)} = tr(1) - 1$ by
Proposition \ref{prop:norm-not-so-basics} part (ii), the result follows from the
observation that $tr_{L'(t)/L(t)}(1) = tr_{L'/L}(1)|_{L(t)}$, i.e. the base
change formula.

(ii) Since $tr(L') - 1 = N(-1)$ we get $(tr(L')- 1)^2 = 1$. The result follows.

(iii) We have $N(2) = N(1 + 1) = 2 + tr(1)$ by Lemma \ref{lemm:wittkop-formula}
again. Thus if we put $\xi = N(2)$ then by
(ii) we find $(\xi - 2)^2 = 2(\xi - 2)$ which implies that $\xi(6-\xi) = 8$.
\end{proof}

\begin{corr} \label{corr:technical-2}
Let $L'/L$ be a degree 2 extension with $char(L) \ne 2$, $x \in GW(L)$ and $2^rx = 0$. Then for $m > 3r$ we have
\[ N_{L'(t_1, \dots, t_m)/L(t_1, \dots, t_m)}(1 + P_mx) = 1 + P_m tr_{L'/L}(x). \]
\end{corr}
\begin{proof}
Using Lemma \ref{lemm:wittkop-formula} we compute that $N(1 + P_m
x) = 1 + tr(P_m x) + N(P_m)N(x)$. Noting that $tr_{L'(t_1, \dots, t_m)/L(t_1,
\dots, t_m)}(P_m x) = P_m tr_{L'/L}(x)$ by the base change and projection
formulas, we thus need to show that $N(P_m)N(x) = 0$.
Since $2^r x = 0$ we get $0 = N(2^r x) = N(2^r) N(x)$, and hence by Lemma
\ref{lemm:technical-1} part (iii) we find that $8^r N(x) = 0$. I claim that
$N(P_m) = (-1)^m 2^{m-1} tr(L') P_m$. Since $m > 3r$ the claim implies that $N(P_m)$
is divisible by $8^r$ and hence $N(P_m)N(x) = 0$ as needed.

To prove the claim, note that $N(P_m) = \prod_{i=1}^m N(\lra{t_i} - 1) =
(-tr(L'))^m P_m$, by Lemma \ref{lemm:technical-1} part (i), and $(-tr(L'))^m =
(-1)^m 2^{m-1} tr(L')$ by Lemma \ref{lemm:technical-1} part (ii).
\end{proof}

\begin{thm}\label{thm:log-module-transfers}
For $n \ge 2$ and $x \in F_n \ul{GW}^\times(X)$, the sequence $\log_{(m)}(x) \in
\ul{I}^n_{tor}(x)$ is eventually constant. Write $\log(x)$ for this eventual
value.
This defines an isomorphisms of homotopy modules $\log: F_* \to
\ul{I}_{tor}^*$ such that for any $n \ge 2$ and any finite separable extension
$L'/L$ with $L$ of
finite transcendence degree over $k$, the following diagram commutes
\begin{equation} \label{eq:compat-tr}
\begin{CD}
F_n GW^\times(L') = F_n(L') @>\log>> I^n_{tor}(L') \\
@VNVV               @V{tr}VV \\
F_n GW^\times(L) = F_n(L) @>\log>> I^n_{tor}(L).
\end{CD}
\end{equation}
Here on the left hand side $N$ denotes Rost's multiplicative transfer,
whereas on the right hand side $tr$ denotes the cohomological transfer present
on any homotopy module, see Appendix \ref{sec:app-htpy-mod}. In particular the cohomological
transfer on $F_n$ coincides with Rost's multiplicative transfer.
\end{thm}
\begin{proof}
It suffices to show the following: (a) for each $n \ge 2$ there is a well-defined morphism of
sheaves of abelian groups $\log: F_n \to \ul{I}^n_{tor}$, (b) the maps $\log$
are compatible with the isomorphisms $F_n \xrightarrow{\iso}
(F_{n+1})_{-1}$ and $\ul{I}^n \xrightarrow{\iso} (\ul{I}^{n+1})_{-1}$. This
implies that there is indeed a morphism of homotopy modules as stated. We then
need to check for each $n \ge 2$ that (c) the map $\log: F_n \to \ul{I}^n_{tor}$ is an
isomorphism, and that it is (d) compatible with the  transfers, in the sense of
equation \eqref{eq:compat-tr}.

For now we shall assume that $vcd_2(k) < \infty$. We will remove this assumption
at the end by a continuity argument. Let us also fix $n \ge 2$, $X \in Sm(k)$ and $r > 0$
such that $2^r \ul{I}^n_{tor}(X) = 0$. Such an $r$ exists by unramifiedness of $\ul{I}^n_{tor}$
and Lemma \ref{lemm:Itors-vanishing}.

\paragraph{Step 1.} \emph{If $m>r$, then $\log_{(m)}: F_*(X) \to
\ul{I}^*_{tor}(X)$ is a homomorphism of abelian groups.}
To see this, suppose that $x, y \in F_n(X) = F_n
\ul{GW}^\times(X)$. Then $x^{P_m} = 1 + P_m x', y^{P_m} = 1 + P_m y'$ for some (unique) $x',
y' \in \ul{I}^n_{tor}(X)$, and by definition $\log_{(m)}(xy) = z'$, for the
unique element $z' \in \ul{I}^n_{tor}(X)$ such that $(xy)^{P_m} = 1 + P_m z'$. We thus
need to show that $(1 + P_m x')(1 + P_m y') = 1 + P_m(x' + y')$.
Note that for any $x', y' \in \ul{I}^n_{tor}(X)$
we have $(1 + P_m x')(1 + P_m y') = 1 + P_m(x' + y') + P_m^2x'y'$. Since
$(\lra{t_i} - 1)^2 = -2(\lra{t_i}-1)$ we find that $P_m^2$ is divisible by $2^m$, and
hence $P_m^2 x' = 0$. This proves the claim.

\paragraph{Step 2.} \emph{For $m > 3r$, the map $\log_{(m)}: F_n(X) \to
\ul{I}^n_{tor}(X)$ is a homomorphism of $\ul{GW}(X)$-modules.}
In other words we need to show that for
$x \in F_n \ul{GW}^\times(X)$
and $y \in \ul{GW}(X)$ we have $y\log_{(m)}(x) =
\log_{(m)}(x^y)$.
By definition we have $x^{P_m} = 1 + P_m \log_{(m)}(x)$. Now
\[ 1 + P_m \log_{(m)}(x^y) = (x^y)^{P_m} = (x^{P_m})^y = (1 + P_m
\log_{(m)}(x))^y \] and so it is enough to show that $(1 + P_m x')^y = 1
+ y P_m x'$ for every $x' \in \ul{I}^n_{tor}(X)$. By unramifiedness, we may
assume that $X$ is the spectrum of a field $L$. Since $GW(L)$ is generated as an
abelian group by $1$ and the traces of quadratic extensions, and $\log_{(m)}$ is
a homomorphism of abelian groups by step 1, it suffices to
prove that for $L'/L$ quadratic we have $(1 + P_m x')^{tr(L')} = 1
+ tr(L') P_m x'$. By the base change and projection formulas, the left hand side equals
$N_{L'(t_1, \dots, t_m)/L(t_1, \dots, t_m)}(1 + P_m x')$.
Since $2^r x' = 0$ by assumption, the claim now follows from Corollary
\ref{corr:technical-2} and the projection formula.

\paragraph{Step 3.} \emph{The sequence $\log_{(m)}(x)$ for $x \in F_n(X)$ is
eventually constant.} Let $m > 3r$. We compute
\[ x^{P_{m+1}} = (x^{P_m})^{\lra{t_{m+1}} - 1}
      = (1 + P_m \log_{(m)}(x))^{\lra{t_{m+1}} - 1}
      = 1 + (\lra{t_{m+1}} - 1) P_m \log_{(m)}(x), \]
where in the last equality we have used step 2.  
In other words we have found that $1 + P_{m+1}\log_{(m+1)}(x) = 1 +
P_{m+1}\log_{(m)} x$. By the uniqueness part
of Lemma \ref{lemm:defn-log-approx}, this means that $\log_{(m+1)}(x) =
\log_{(m)}x$, which proves the claim.

We conclude that the limiting map $\log: F_n \to \ul{I}^n_{tor}$ exists. It is
easy to check that it is a homomorphism of (pre)sheaves of sets. By step 1 it is a
homomorphism of sheaves of abelian groups.

\paragraph{Step 4.} \emph{The limiting map $\log: F_n(X) \to \ul{I}^n_{tor}(X)$
is compatible with transfers in the sense of equation \eqref{eq:compat-tr}.}
By
unramifiedness we may assume that $X$ is the spectrum of a field.
We need to prove that two elements in the lower right hand corner of
equation \eqref{eq:compat-tr} are equal. I claim that we may assume that $L' =
\prod_i L_i$, where each $L_i$ is an iterated quadratic extension over $L$.
For this we use that (a) $\log$ is a map of presheaves, (b) transfers commute
with base change, (c) $I^n(L) \to I^n(L')$ is injective if $[L':L]$ is odd (by
Lemma \ref{lemm:odd-degree-base-change}), and then apply the argument from the
proof of Proposition \ref{prop:norm-basics}.
Now that we have reduced to $L' = \prod_i L_i$,
the result follows from Corollary \ref{corr:technical-2}.

\paragraph{Step 5.} \emph{$\log$ is a morphism of homotopy modules.}
In other words $\log$ is
compatible with the isomorphisms $\beta_n^\dagger: F_n \to (F_{n+1})_{-1}$. To
see this it is enough to show that $\log([u] x) = [u] \log(x)$. But
\[ \log([u]x) = \log(x^{\lra{u}-1}) = (\lra{u}-1)\log(x) = [u]\log(x). \]
Here the first and last equality are by definition, and the middle one is
because $\log$ is a homomorphism of $GW$-modules, by step 2.

\paragraph{Conclusion of proof for $vcd_2(k) < \infty$.}
We have already established (a), and (b) is step 5.
It is clear that for each $X$, $\log(X)$ is a bijection of sets,
because each of the maps $\log_{(m)}(X)$ is a bijection. This proves (c).
Property (d) was established in step 4.

\paragraph{Step 6.} \emph{If $L$ is any field of characteristic $\ne 2$, and $x
\in F_n GW^\times(L)$, then the sequence $\log_{(m)}(x)$ is eventually
constant.}
By Corollary \ref{corr:app-cont}
there exist a subfield $l \subset L$ which is finitely generated over the prime
field and $y \in F_n GW^\times(l)$, such that $x = y|_L$. By step 3,
$\log_{(m)}(y)$ is eventually constant, and hence so is $\log_{(m)}(x) =
\log_{(m)}(y|_L) = \log_{(m)}(y)|_L$.

\paragraph{Conclusion for general $k$.}
We use the standard continuity argument. It follows from step 6 that $\log: F_* \to
\ul{I}^*_{tor}$ is a well-defined map (of sheaves of sets, say). Then by
Corollary \ref{corr:app-cont} from Appendix \ref{sec:app-cont} we have
$\log(L) = \colim_l \log(l)$, where the colimit is over
$l \subset L$ which are finitely generated over the prime field. Thus $vcd_2(l)
< \infty$ for all such $l$, and hence $\log(l)$ is an isomorphism with all
the desired properties (respects module structures and transfers), since we have
already established the theorem in this case. It follows that the colimit
$\log(L)$ has all the desired properties, too.

This concludes the proof.
\end{proof}

\begin{rmk}
Let $k$ be a not necessarily perfect field and $p: Spec(k) \to Spec(k_0)$ be a
morphism to a perfect field. Then $F_* \iso p^*(F_*|_{k_0})$ and $\ul{I}^*_{tor}
\iso p^*(\ul{I}^*_{tor}|_{k_0})$, by the standard continuity arguments. It
follows that the map $\log: F_* \to \ul{I}^*_{tor}$ is a well-defined
isomorphism of homotopy modules over $k$. If $G_*$ is a homotopy module over $k$
and $k$ is not perfect, then it is not clear (to the author) how to define
transfers on $G_*$. If however $G_*$ is a homotopy module over
$k_0$ then $p^* G_*$ is a homotopy module over $k$ which has canonical
transfers. This applies to $F_*$ and $\ul{I}^*_{tor}$, and we thus see that all
the claims of Theorem \ref{thm:log-module-transfers} make sense and hold over
$k$.
\end{rmk}

\section{Delooping $\ul{GW}^\times$}
\label{sec:delooping}

In this section, we finally put everything together: we shall construct a
homotopy module $T_*$ such that $T_0 = \ul{GW}^\times$. We do this by taking the
presentation
\[ I^2(k) \xrightarrow{a} GW(k)/2 \oplus F_2 GW^\times(k) \to GW^\times(k) \to 1 \]
from Proposition \ref{prop:GW-times-presentation} and exhibiting a map
$\tilde{a}: \ul{I}^{*+2} \to \ul{K}_*^{MW}/2 \oplus F_{*+2}$ of homotopy modules
such that $a = (\tilde{a})_0$.

\begin{lemm} \label{lemm:2-vanishing}
For any field $L$ of characteristic different from 2
we have the equality $2(\lra{2} - 1) = 0 \in GW(L)$.
\end{lemm}
\begin{proof}
Since $(x+y)^2 + (x-y)^2 = 2(x^2 + y^2)$ we have $\lra{1,1} = \lra{2,2}$. The
result follows.
\end{proof}

\begin{lemm} \label{lemm:identify-1-exp}
The maps $\ul{I}^n \to F_n \ul{GW}^\times, x \mapsto (-1)^x$ assemble into a
morphism of homotopy modules $\ul{I}^* \to F_*$. The composite
\[ \ul{I}^* \to F_* \xrightarrow{\log} \ul{I}^*_{tor} \]
is multiplication by $\lra{2} - 1$.
\end{lemm}
Let us note that multiplication by $\lra{2} - 1$ is the zero map if $\sqrt{2}
\in k$. Note also that in any case $\lra{2} - 1 \in GW(k)_{tor}$ by Lemma
\ref{lemm:2-vanishing}.
\begin{proof}
We have $(-1)^{\ul{I}^n} \subset F_n \ul{GW}^\times$ by Proposition
\ref{prop:norm-not-so-basics} part (iv), so the map exists as claimed. In order
to see that we have a morphism of homotopy modules, we need to show that for $x
\in \ul{I}^n(L)$ and $u \in L^\times$ we have $(-1)^{[u]x} = [u]((-1)^x)$. Now
$[u]x = (\lra{u} - 1)x$ by the definition of $\ul{I}^*$ and so $(-1)^{[u]x} =
((-1)^x)^{\lra{u} - 1}$, which equals
$[u]((-1)^x)$ by the definition of $F_*$ (see Corollary \ref{corr:F*-existence}), as needed.

Recall that $\ul{I}^* \iso \ul{K}_*^W = \ul{K}_*^{MW}/h$, cf.
\cite{morel2004puissances}. Consider the composite map of homotopy modules
$\ul{K}^{W}_* \iso \ul{I}^* \to F_* \xrightarrow{\log} \ul{I}^*_{tor}$. Let
$x \in I^0_{tor}(k) = W(k)_{tor} \subset W(k)$ be the image of $1 \in
K^{W}_0(k)$. I claim that the composite $f: \ul{I}^* \to F_* \xrightarrow{\log}
\ul{I}^*_{tor}$ is given by multiplication by $x$. Indeed if $y \in K^{MW}_*(L)$
then $f(y \cdot 1) = y \cdot f(1) = y \cdot x$, this being a map of homotopy
modules. Since $\ul{K}^{MW}_* \to \ul{I}^*$ is surjective, this proves the
claim. It remains to show that $x = \lra{2} - 1$. For this it is enough to prove that the composite $\ul{I}^2 \to F_2
\ul{GW}^\times \to \ul{I}^2_{tor}$ is given by multiplication by $\lra{2} - 1$.
Indeed then the two-fold contraction $\ul{W} \to \ul{W}_{tor}$ will also be
given by multiplication by $\lra{2} - 1$, so in particular $x = f(1) = \lra{2} -
1$.

Fix some field $L$. The ideal $I^2(L)$ is generated by elements of the form
$(\lra{a} - 1)(\lra{b} - 1)$. It is thus enough to prove that
\begin{equation} \label{eqn:to-show-2}
   (-1)^{(\lra{a} - 1)(\lra{b} - 1)(\lra{t_1} - 1) \cdots (\lra{t_m} - 1)} = 1 +
  (\lra{2} - 1)(\lra{a} - 1)(\lra{b} - 1)(\lra{t_1} - 1) \cdots (\lra{t_m} - 1)
\end{equation}
for all $m$. We prove this in two steps. First we deal with $m=0$.
Let $A = k(\sqrt{a}), B = k(\sqrt{b})$. Then we compute
\begin{align*}
(-1)^{(\lra{a} - 1)(\lra{b}-1)} &= (-1)^{(\lra{a} + 1)(\lra{b} + 1)}
                  & (-1)^x = (-1)^{-x} \\
        &= (-1)^{tr(A)tr(B)}     & \text{($T_A$), ($T_B$), $1 = \lra{2}^2$} \\
        &= (tr(A) - 1)^{tr(B)}   & \text{Proposition \ref{prop:norm-not-so-basics} part (ii)}\\
        &= (\lra{2}(\lra{a} + 1) - 1)^{tr(B)}  & (T_A) \\
        &= (\lra{a} + 1)^{tr(B)} - tr(A)tr(B) + (-1)^{tr(B)}
                                 & \text{(*), (**), (***)} \\
        &= 2 + \lra{a}tr(B) - tr(A)tr(B) + tr(B) - 1 &
                   \text{(**), (***), Proposition \ref{prop:norm-not-so-basics} part (ii)} \\
        &= 1 + tr(B)(1 + \lra{a})(1 - \lra{2}) & (T_A) \\
        &= 1 + (\lra{2} - 1)(\lra{a} - 1)(\lra{b} - 1)
                    & (T_B), \text{(****)} \\
        &=: 1 + \xi.
\end{align*}
Here the right hand column justifies each manipulation, with the following
abbreviations: ($T_A$) means we use that $tr(A) = \lra{2}(\lra{a} + 1)$, ($T_B$)
means we use that $tr(B) = \lra{2}(\lra{b} + 1)$, (*) means we use that $(xy)^z = x^z y^z$,
(**) means we use that \text{$\lra{x}^{tr(B)} = 1$ as follows from \cite[Lemma
2.6(ii)]{wittkop2006multiplikative}}, (***) means we apply Lemma
\ref{lemm:wittkop-formula}, in the form that $(x + y)^{tr(B)} = x^{tr(B)} +
y^{tr(B)} + tr(B)xy$, and (****) means we use that $x(\lra{2} - 1) = -x(\lra{2} -
1)$, which follows from Lemma \ref{lemm:2-vanishing}. This proves the claim when
$m=0$.

In order to prove the claim for $m>0$, it is enough to show that
for any field extension $L'/L$, and any $y \in GW(L')$
we have $(1 + \xi)^y = 1 +
y\xi$. Suppose that $(1 + \xi)^{y_i} = 1 + y_i\xi$, for $y_1,
y_2 \in GW(L')$. Then $(1+\xi)^{y_1 + y_2} = 1 + (y_1 + y_2)\xi + y_1y_2 \xi^2 =
1 + (y_1 + y_2)\xi$, provided that in addition $\xi^2 = 0$. Note that this
condition is equivalent to $(1 + \xi)^2 = 1 + 2\xi$.
Since $GW(L')$ is generated as an abelian group by $1$ and the traces
of quadratic extensions,
it thus suffices to show that $(1 +
\xi)^{tr(T)} = 1 + tr(T)\xi$ for $T/L'$ degree 2, or equivalently that
$\xi^{tr(T)} = 0$.
We have $(\lra{a} - 1)^{tr(T)} = 1 - tr(T)\lra{a} + tr(T) - 1 =
-tr(T)(\lra{a} - 1)$, by (***) with $T$ in place of $B$, and similarly for $b$
or $2$ in place of $a$.
Hence \[ \xi^{tr(T)} = (\lra{2} - 1)^{tr(T)} (\lra{a} -
1)^{tr(T)} (\lra{b} - 1)^{tr(T)} = -tr(T)^3\xi = -4tr(T)\xi = 0, \] by Lemmas
\ref{lemm:2-vanishing} and \ref{lemm:technical-1}(ii).

This concludes the proof.
\end{proof}

\begin{thm} \label{thm:final}
Let $k$ be a perfect field of characteristic different from $2$. There exists a
short exact sequence of homotopy modules
\[ 0 \to \ul{K}^W_{*+2}/[-1]\ul{K}^W_{*+1} \xrightarrow{\eta^2 \oplus (\lra{2} - 1)}
   \ul{K}_*^{MW}/2 \oplus \ul{K}^W_{tor, *+2} \to T_* \to 0, \]
defining $T_*$. There is a canonical isomorphism $T_0 \iso \ul{GW}^\times$. Via
this isomorphism, the
$\ul{GW}$-module structure as well as the cohomological transfers on $T_0$
coincide with the module structure from Section \ref{sec:sheaf-GW-times} and
Rost's multiplicative transfers, respectively.
\end{thm}
Here by $[-1] \in K_1^W(k) \iso I(k)$ we denote the element $\lra{-1} - 1$, and
$\eta^2$ is just the natural inclusion $\ul{K}_{*+2}^W = \ul{I}^{*+2}
\hookrightarrow \ul{K}_*^{MW} = \ul{I}^* \times_{\ul{k}_*^M} \ul{K}_*^M$. One
may note that $\lra{-1} - 1 \in I(k) \subset W(k)$ is the same as $-2 \in W(k)$.
\begin{proof}
Note that $\eta^2([-1]) = \eta(-2)
= -2\eta = 0 \in \ul{K}_*^{MW}/2$, so the first map is well-defined, and that
$(\lra{2} - 1)[-1] = (\lra{2}-1)(\lra{-1} - 1) = (\lra{2}-1)(-2) = 0 \in
I(K) \subset W(K)$ by Lemma \ref{lemm:2-vanishing}, so the second map is
well-defined.

Being a sequence of homotopy modules, exactness in all degrees is equivalent to exactness in
all sufficiently high degrees.
Let us first show that $\ul{K}^W_{n+2}/[-1] \xrightarrow{\eta^2 \oplus (\lra{2} - 1)}
\ul{K}^{MW}_n/2 \oplus \ul{K}^W_{tor, n+2}$ is injective, for $n \ge 0$. For this let $x \in
K^W_{n+2}(L)$
and suppose that $\eta^2(x) = 0$. Equivalently, we have $x \in
I^{n+2}(L)$ with $\eta^2(x) \in 2K_n^{MW}(L)$. Then in particular $x \in 2I^n(L) \subset
2W(L)$. By \cite[Theorem 2.2]{ARASON2001150} we conclude that $x \in
2I^{n+1}(L)$. Note that this reference indeed applies, since $-2 = \langle -1, -1 \rangle \in W(L)$ is a Pfister
form. Consequently $x \equiv 0 \in K^W_{n+2}(L)/[-1]K^W_{n+1}(L)$, i.e. $\eta^2$
alone is already injective.

I claim that we have, for every field $L/k$, a commutative diagram as follows
\begin{equation} \label{eq:helper-diagram}
\begin{CD}
K^W_{2}(L)/[-1]K^W_{1}(L) @>{\eta^2 \oplus (\lra{2} - 1)}>>
       K_0^{MW}(L)/2 \oplus K^W_{tor, 2}(L)  @>>> T_0(L)    \\
@AaAA                @Ab\oplus \log AA                                \\
I^2(L)                    @>{p \oplus q}>> GW(L)/2 \oplus F_2 GW^\times(L) @>{r/s}>> GW^\times(L).
\end{CD}
\end{equation}
Here we have used the notation of Proposition \ref{prop:GW-times-presentation}.
The map $a$ is induced from the isomorphism $I^2(L) \iso K^W_{2}(L)$, and the map
$b$ is $GW(L)/2 \iso K_0^{MW}/2$. The claim is an immediate consequence of Lemma
\ref{lemm:identify-1-exp}.

Consider the map
$r'/s': \ul{K}^{MW}_0/2 \oplus \ul{K}^W_{tor, 2} \iso \ul{GW}/2 \oplus F_2
\ul{GW}^\times \to \ul{GW}^\times$. Here $F_2 \ul{GW}^\times \iso
\ul{I}^2_{tor} \iso \ul{K}^W_{tor, 2}$ is the composite of the logarithm
isomorphism and Morel's isomorphism $\ul{I}^* \iso \ul{K}^W_*$ restricted to
torsion, $r': \ul{GW}/2 \to
\ul{GW}^\times$ is $x \mapsto (-1)^x$ and $s': F_2 \ul{GW}^\times \to
\ul{GW}^\times$ is the inclusion. I claim that $r'/s'$ factors through
$\ul{K}^{MW}_0/2 \oplus \ul{K}^W_{tor, 2} \to T_0$, and induces an isomorphism
$T_0 \iso \ul{GW}^\times$. Since we are dealing with unramified sheaves, both
claims may be checked on sections over fields. There they follow from the
commutativity of diagram
\eqref{eq:helper-diagram}, using that the map $a$ is surjective, $b \oplus
\log$ is an isomorphism, and $r = r'(L), s = s'(L)$ under this isomorphism.

Now in order to see that $r'/s'$ preserves the $\ul{GW}$-module structure and the
transfers, it suffices to consider $r'$ and $s'$ separately. The fact that this
works for $r'$ follows from Proposition \ref{prop:GW-module-existence-basics}
parts (ii) and (iv). The map $s'$ is the composite of the logarithm isomorphism,
which preserves the module structure and transfers by Theorem
\ref{thm:log-module-transfers}, and the morphism of multiplication by a constant
(namely $\lra{2} - 1$) which also preserves the module structure and transfers.
This concludes the proof.
\end{proof}

\begin{rmk}
If $\sqrt{2} \in k$ (i.e. $\lra{2} = 1 \in W(k)$)
then we get a splitting $T_* \iso \ul{K}^{MW}_*/(2,\eta^2) \oplus
\ul{K}^W_{tor,*+2}$, and in particular $GW^\times(k) \iso GW(k)/(2, I^2) \oplus
I^2_{tor}(k)$, but not in general. This is essentially the same as
Remark \ref{rmk:sqrt-2}.
\end{rmk}

\appendix

\section{Recollections on homotopy modules}
\label{sec:app-htpy-mod}
Throughout, $k$ is a perfect base field. We recall some well-known facts about
homotopy modules which seem hard to find explicitly in the literature. We make
no claim to originality. Throughout $F_*$ will be an arbitrary homotopy module,
with no necessary relation to the specific homotopy module constructed in
Section \ref{sec:sheaf-GW-times}.

\paragraph{The basics.} 
Recall that for any (pre)sheaf $F$ (on $Sm(k)$), its \emph{contraction} is
\[ F_{-1} := \iHom(\ZZ[\Gm], F) \in Pre(Sm(k)). \]
Here $\ZZ[\Gm] = \ZZ[(\Aone \setminus 0, 1)]$, where for a pointed scheme $(X,
x)$ we put $\ZZ[(X, x)] = \ZZ[X]/\ZZ[\{x\}]$.
Recall moreover that $F$ is called strictly homotopy invariant if for all $X \in
Sm(k)$ and $n \ge 0$, the canonical map $H^n_{Nis}(X, F) \to H^n_{Nis}(X \times
\Aone, F)$ is an isomorphism.

A \emph{homotopy module} \cite[Section
5.2]{morel-trieste} consists of a collection of strictly homotopy invariant
sheaves $F_* \in Shv_{Nis}(Sm(k)), * \in \ZZ$ together with isomorphisms $F_n
\to (F_{n+1})_{-1}$. A morphism of homotopy modules $\alpha_*: F_* \to G_*$ consists of
morphisms of sheaves $\alpha_n: F_n \to G_n$ for all $n$ such that
the following diagram commutes for each $n$
\begin{equation*}
\begin{CD}
F_n @>>> (F_{n+1})_{-1} \\
@V{\alpha_n}VV  @VV{(\alpha_{n+1})_{-1}}V \\
G_n @>>> (G_{n+1})_{-1}.
\end{CD}
\end{equation*}

The category of homotopy modules is equivalent to the heart of the homotopy
$t$-structure on $\SH(k)$ \cite[Theorem 5.2.6]{morel-trieste}. This implies that
they have a lot more structure than is immediately apparent. In this appendix we
clarify some of this structure.

\paragraph{$\ul{K}_*^{MW}$-module structure.} The object $\ul{\pi}_0(\tunit)_*
\iso \ul{K}_*^{MW}$ is the unit of a symmetric monoidal structure on the
category of homotopy modules. For the definition of the sheaf of unramified
Milnor-Witt $K$-theory $\ul{K}_*^{MW}$, see \cite[Chapter 3]{A1-alg-top}. Its
sections over a field $L$ are generated by the classes $[u] \in K_1^{MW}(L)$ for
$u \in L^\times$ and $\eta \in K_{-1}^{MW}(L)$. One puts $\lra{u} = 1 +
\eta[u]$; this induces an isomorphism $K_0^{MW}(L) \iso GW(L)$ \cite[Lemma
3.10]{A1-alg-top}. For each $n \ge 1$, the induced map $\ZZ[\Gm]^{\otimes n} \to
\ul{K}_n^{MW}, u_1 \otimes \dots \otimes u_n \mapsto [u_1] \cdots [u_n]$ induces
a surjection on sections over fields \cite[Lemma 3.6]{A1-alg-top}.

Now suppose that $F_*$ is a homotopy module. The isomorphism $F_n \to
(F_{n+1})_{-1}$ corresponds by adjunction to a morphism $\ZZ[\Gm] \otimes F_n \to
F_{n+1}$. Using the identification of the category of homotopy modules with the
heart of $\SH(k)$, one may show that
it factors through the surjection $\ZZ[\Gm] \to \ul{K}_1^{MW}$. By
contraction, the pairing $\ul{K}_1^{MW} \otimes F_n \to F_{n+1}$ induces
$(\ul{K}_1^{MW})_{-2} \otimes F_n \to (F_{n+1})_{-2}$ and hence a multiplication
$\eta: F_n \to F_{n-1}$. Since $\ul{K}_*^{MW}$ is generated by $\ul{K}_1^{MW}$
and $\eta \in \ul{K}_{-1}^{MW}$, it follows that there is at most one extension
to a pairing $\ul{K}_n^{MW} \otimes F_m \to F_{n+m}$ for all $m,n\in \ZZ$.
It is a consequence of the
identification of the category of homotopy modules with the heart of $\SH(k)$
that this extension always exists.

\paragraph{Cohomological transfers.} A homotopy module automatically has
cohomological transfers, i.e. for any finite étale morphism $f: X \to Y$ with
$X, Y$ essentially smooth over $k$, there is a transfer $tr_f: F_*(X) \to
F_*(Y)$. See \cite[Corollary 5.30]{A1-alg-top} or \cite[Section
4]{bachmann-real-etale}. Cohomological transfers are functorial in morphisms of
homotopy modules and satisfy the projection and base change formulas (\emph{loc.
cit.}). This means the following. If $a \in \ul{K}_*^{MW}(X)$ and $m \in
F_*(Y)$, then $tr_f(a f^* m) = tr_f(a) \cdot m$, and similarly if $b \in
\ul{K}_*^{MW}(Y)$ and $n \in F_*(X)$ then $tr_f(f^*(b) n) = b tr_f(n)$; these
are the projection formulas. Moreover if we have a cartesian square
\begin{equation*}
\begin{CD}
X' @>g'>> X \\
@Vf'VV   @VfVV \\
Y' @>g>>  Y
\end{CD}
\end{equation*}
with $Y, Y'$ smooth and $f$ étale, then $g^* tr_f(n) = tr_{f'} g'^*(n)$; this is
the base change formula.

\paragraph{More on contractions.}
\begin{lemm} \label{lemm:app-contractions}
Let $t$ be a coordinate on $\Aone$ and $F_*$ a homotopy module. Write $i_1:
Spec(k) \to \Aone \setminus 0$ for the inclusion of the point $1$.
Then for $X
\in Sm(k)$ we have $F_*(X \times (\Aone \setminus 0)) \iso F_*(X) \oplus
F_{*-1}(X)$. Here the map $F_{*-1}(X) \to F_*(X \times (\Aone \setminus 0))$
is multiplication by $[t] \in
K_1^{MW}(k[t, t^{-1}])$, the map $F_*(X) \to F_*(X \times (\Aone \setminus 0))$ is pullback along the
canonical projection, and the subgroup $F_{*-1}(X) \subset F_*(X \times (\Aone \setminus
0))$ is precisely the kernel of $i_1^*$.
\end{lemm}
\begin{proof}
We have the inclusions $Spec(k) \xrightarrow{i_1} \Aone \setminus 0 \to \Aone$.
Since $F$ is homotopy invariant, it follows that $F_*(X \times (\Aone \setminus
0)) = F_*(X) \oplus M_*(X)$, where $M_*(X)$ is the kernel of $i_1^*$ and by
definition of contraction, $M_*(X) = (F_*)_{-1}(X)$. We thus use the defining
isomorphism of a homotopy module to identify $M_*(X) \iso F_{*-1}(X)$. It
remains to see that this isomorphism is given by multiplication by $[t]$.

For this, let $F$ be any homotopy invariant sheaf and $U \in Sm(k)$ be arbitrary.
If $X \in Sm(k)$ then $\iHom(\ZZ[X], F)(U) = F(X \times U)$ and the pairing
$\Hom(U, X) \times \iHom(\ZZ[X], F)(U) \to F(U)$ is given by $\alpha, x \mapsto
\alpha'^*(x)$, where $\alpha' = (\alpha, \id_U): U \to X \times U$ and $x \in
F(U \times X)$. Since $F_{-1}(U) = \iHom(\ZZ[\Gm], F)(U) \subset F((\Aone
\setminus 0) \times U)$,
the pairing $\Gm(U) \otimes
F_{-1}(U) \to F(U)$ is given by $(u, s) \mapsto u'^*s$. Here $s \in F_{-1}(U)
\subset F(U \times (\Aone \setminus 0)), u \in \Hom_k(U, \Aone \setminus 0)$
and $u' = (u, \id_U) \in \Hom_U(U, (\Aone \setminus 0) \times U)$.
Consequently ``multiplication by
$[t]$''
\[ F(X \times (\Aone \setminus 0))
     \supset F_{-1}(X) \to F_{-1}(X \times (\Aone \setminus 0))
      \xrightarrow{[t]} F(X \times (\Aone \setminus 0)) \]
is given by $F(X \times (\Aone \setminus 0)) \supset F_{-1}(X) \ni s \mapsto
t^* p^* (s)$, where $p: X \times (\Aone \setminus 0)^2 \to X \times (\Aone
\setminus 0)$ is projection to the first two factors and $t: X \times (\Aone
\setminus 0) \to X \times (\Aone \setminus 0)^2$ is $(x, u) \mapsto (x, u, u)$.
Thus $pt = \id$ and so multiplication by $[t]$
corresponds to the canonical inclusion, as was to be shown.
\end{proof}

\paragraph{Boundary maps.} Being strictly homotopy invariant, each of the
sheaves $F_n$ is unramified \cite[Lemma 6.4.4]{morel2005stable}. This means that
for a dense open immersion $U \to X \in Sm(k)$, the restriction
$F_*(X) \to F_*(U)$ is injective and that moreover for $X$ connected we have
\[ F_*(X) = \bigcap_{x \in X^{(1)}}F_*(X_x). \]
Here $X^{(1)}$ denotes the set of points of codimension one,
$F_*(X_x)$ denotes the stalk at $x$, and the intersection takes place in
$F_*(k(X))$.

If $L/k$ is a field extension, then by a standard colimit
procedure there is a well-defined group of sections $F_*(L)$. More generally
this is true if $L$ is a scheme which is a filtering inverse limit of a system
of smooth schemes with affine transition morphisms. If $L$ is a finitely
generated field extension of $k$ and $\mathcal{O} \subset L$ is a geometric dvr,
then by definition there exist $X \in Sm(k)$ and $x \in X^{(1)}$ such that $L
\iso k(X)$ and $Spec(\mathcal{O})$ is the localization of $X$ in $x$.
Let $\kappa$ be the residue field
of $\mathcal{O}$. Then for every choice of uniformizer $\pi$ of $\mathcal{O}$
there exists a canonical \emph{boundary map}
\[ \partial^\pi: F_*(L) \to F_{*-1}(\kappa) \]
with kernel $F_*(\mathcal{O})$ \cite[discussion after Corollary
2.35]{A1-alg-top}.

In this situation, we write $s: Spec(\kappa) \to Spec(\mathcal{O})$ for the
inclusion of the closed point.

\begin{lemm} \label{lemm:app-bdry-formula}
For $u \in \mathcal{O}^\times$ and $m \in F_*(\mathcal{O})$ we have
$\partial^\pi([u\pi]m) = \lra{s^*(u)}s^*(m)$.
\end{lemm}
\begin{proof}
First note that for $x \in K_*^{MW}(K)$ and $m \in F_*(\mathcal{O})$ we have
$\partial^\pi(xm) = \partial^\pi(x)s^*(m)$. To see this, one goes
back to the geometric construction of $\partial^\pi$ as in \cite[Corollary
2.35]{A1-alg-top}. That is we observe that after canonical identifications,
$\partial^\pi$ corresponds to the boundary map $\partial: H^0(K, F_*) \to
H^1_v(\mathcal{O}, M_*)$ in the long exact sequence for cohomology with
support. Our claim then follows from the observation that for any sheaf of rings
$K$ on a space $X$ and $K$-module $F$,
the boundary map in cohomology with support satisfies
our claim. To see this, just note that multiplication by $m \in F(X)$ induces a
homomorphism of sheaves $K \to F \in Shv(X)$ and consider the induced
homomorphism of long exact sequences for cohomology with support.

We also have $\partial^\pi([u\pi]) = \partial^\pi(\lra{u}[\pi] + [u]) =
\lra{s^*(u)}$, using \cite[Lemma 3.5(1) and Proposition 3.17(3)]{A1-alg-top} and
the fact that the homomorphism $\partial^\pi$ has kernel
$\ul{K}_*^{MW}(\mathcal{O})$.

This concludes the proof.
\end{proof}

\begin{lemm} \label{lemm:app-submodule}
If $G_* \hookrightarrow F_*$ is an inclusion of homotopy modules, then for any
connected $X \in Sm(k)$ we have $G_*(X) = F_*(X) \cap G_*(k(X))$.
\end{lemm}
\begin{proof}
Since $G_*$ is unramified we have $G_*(X) = \bigcap_{x \in X^{(1)}} G_*(X_x)$.
It thus suffices to prove the lemma in case that $X = Spec(\mathcal{O})$ with
$\mathcal{O} \subset L$ a dvr with uniformizer $\pi$.
Note that once a uniformizer $\pi$ has been fixed, the construction of the
boundary map $\partial^\pi$ of a homotopy module is completely canonical,
which implies that $\partial^\pi_G =
\partial^\pi_F|_G$ and hence $G_*(\mathcal{O}) = ker(\partial^\pi_G) =
ker(\partial^\pi_F) \cap G_*(L)$. This concludes the proof.
\end{proof}

\section{Recollections on continuity}
\label{sec:app-cont}
In this section we collect some continuity results which we use repeatedly, but
could not find any references for. We again make no claims to originality.

By an essentially smooth $S$-scheme $X$ we mean a cofiltered diagram of
$S$-schemes $X_\alpha, \alpha \in \Lambda$ with each $X_\alpha \to S$ smooth, and each
transition map $X_\alpha \to X_\beta$ affine. By abuse of notation, we denote
the limit $\lim_\alpha X_\alpha$ also by $X$. We call a morphism of schemes $X \to S$
essentially smooth if it can be obtained as the limit of a cofiltered
diagram as above.

By an essentially finite type $S$-scheme $X$ we mean the same thing, except that
$X_\alpha \to S$ is required to be of finite type instead of
smooth.

Note that it follows from \cite[Théorème 8.8.2(2) and Théorème 8.10.5(v)]{EGAIV}
\cite[Tags 0C0C and 01OY]{stacks-project} that essentially
smooth (respectively essentially finite type) morphisms between Noetherian
schemes are stable under composition.

\begin{lemm}\label{lemm:witt-continuity}
Suppose $X \to S$ is an essentially finite type morphism of Noetherian schemes.
Then
\[ GW(X) \iso \colim_\alpha GW(X_\alpha) \]
via the pullbacks $GW(X_\alpha) \to GW(X)$.
\end{lemm}
\begin{proof}
By \cite[Tag 01ZR]{stacks-project} the category of coherent sheaves on $X$ is
the colimit of the categories of coherent sheaves on the $X_\alpha$. Any open
subscheme of $X$ is the base change of an open subscheme of $X_\alpha$ for
$\alpha$ sufficiently large \cite[Théorème 8.8.2(2) and Théorème 8.10.5(iii)]{EGAIV} and
hence it follows easily that the category of vector bundles (locally free finite
rank sheaves) on
$X$ is also the colimit of the categories of vector bundles on the $X_\alpha$.
The same result for the categories of bilinear bundles is now formal, and then
$K(Bil(X)) = \colim_i K(Bil(X_\alpha))$.

By definition we have $GW(X) = K(Bil(X))/J(X)$, where $J(X)$ is the ideal
consisting of elements $V - W$, where $V, W$ range over metabolic bilinear
bundles with isomorphic Lagrangians $L$ \cite[Section I.4]{knebusch-bilinear}.
Recall that $L \subset V$ being a Lagrangian means that $V = L \oplus L^\perp$.

It remains to show that $J(X) = \colim_i J(X_\alpha)$. This is immediate from the
description of the category $Bil(X)$ as the colimit of the categories
$Bil(X_\alpha)$.
\end{proof}

\begin{lemm}\label{lemm:Nis-cont}
Let $s: X \to S$ be an essentially smooth morphism between Noetherian schemes of finite
dimension, and $F \in Pre(Sm(S))$ a presheaf.
For each $\alpha$ let $s_\alpha: X_\alpha \to S$ be the structure map. Then
\[ (s^* F)(X) = \colim_\alpha F(X_\alpha). \]
Moreover, $s^*$ preserves Nisnevich sheaves.
\end{lemm}
\begin{proof}
This is a very special case of \cite[Lemmas A.3 and A.4]{hoyois-algebraic-cobordism}.
\end{proof}

\begin{corr} \label{corr:app-cont}
If $K$ is a field, then $GW(K) = \colim_k GW(k)$, where $k$ runs through
the subfields of $K$ which are finitely generated over the prime subfield. Such
$k$ in particular have finite virtual 2-étale cohomological dimension.

More generally, let $p: X \to S$ be an essentially smooth morphism between
Noetherian schemes of finite dimension. Then there is a canonical isomorphism $p^*
\ul{GW} \iso \ul{GW}$ in $Shv(Sm(X)_{Nis})$.

The same is true for $I^n$ or $W$ in place of $GW$.
\end{corr}
\begin{proof}
Since $K = \bigcup_k k$ we have $Spec(K) = \lim_k Spec(k)$. Also the system is
filtering with affine transition morphisms. Hence the first claim follows from
Lemma \ref{lemm:witt-continuity}. For the claim about cohomological dimension,
see \cite[Theorem 28 of
Chapter 4]{shatz1972profinite}.

For the more general statement, we note that the statement with the presheaf
$GW$ in place of the sheaf $\ul{GW}$ follows from Lemmas
\ref{lemm:witt-continuity} and \ref{lemm:Nis-cont}.
So we need to show that $p^*$ commutes with taking
the associated sheaf. This follows from \cite[Tag 00WY]{stacks-project}.

Filtered colimits of abelian groups are exact, so the case of $I^1 =
ker(GW \to \ZZ)$ and $W = coker(\ZZ \to GW)$ follow from
$GW$. For any
filtering system $(R_\alpha, I_\alpha)$ of rings with a specified ideal we get
$\colim_\alpha I_\alpha^n \iso (\colim_\alpha I_\alpha)^n$, and hence we have
established the claim about $I^n$. Finally the claims about $\ul{I}^n,
\ul{W}$ are deduced from the results for $I^n, W$ as before.
\end{proof}

\begin{rmk}
Suitably formulated, the results in this section hold in
much greater generality. We do not need this in the present article, so avoid
the extra complications.
\end{rmk}

\bibliographystyle{plainc}
\bibliography{bibliography}

\end{document}